\documentclass[11pt]{article}
\usepackage{amsfonts}
\usepackage{color}
\usepackage{amsmath,amssymb,comment}
\newtheorem{theo}{Theorem}[section]
\newtheorem{lem}[theo]{Lemma}
\newtheorem{cor}[theo]{Corollary}

\newtheorem{defi}[theo]{Definition}

\newcommand{\mysection}[1]{\setcounter{equation}{0}\section{#1} }
\newcommand{\proof}{{\sc Proof.} \quad}
\newcommand{\proofc}{{\sc Proof} \ }
\newcommand{\be}{\begin{equation} \label}
\newcommand{\ee}{\end{equation}}
\newcommand{\bea}{\begin{eqnarray}\label}
\newcommand{\eea}{\end{eqnarray}}
\newcommand{\bas}{\begin{eqnarray*}}
\newcommand{\eas}{\end{eqnarray*}}
\newcommand{\bit}{\begin{itemize}}
\newcommand{\eit}{\end{itemize}}
\newcommand{\qed}{\hfill$\Box$ \vskip.2cm}
\newcommand{\nn}{\nonumber}
\newcommand{\R}{\mathbb{R}}
\newcommand{\N}{\mathbb{N}}
\newcommand{\pO}{\partial\Omega}

\newcommand{\eps}{\varepsilon}

\newcommand{\wto}{\rightharpoonup}
\newcommand{\wsto}{\stackrel{\star}{\rightharpoonup}}
\newcommand{\hra}{\hookrightarrow}
\newcommand{\io}{\int_\Omega}

\newcommand{\del}{\delta}
\newcommand{\al}{\alpha}
\newcommand{\vt}{\vartheta}
\newcommand{\lam}{\lambda}

\newcommand{\sig}{\sigma}
\newcommand{\pa}{\partial}
\newcommand{\bom}{\overline{\Omega}}
\newcommand{\Om}{\Omega}

\newcommand{\ov}{\overline}

\newcommand{\wh}{\widehat}
\newcommand{\wt}{\widetilde}

\newcommand{\hs}{\hspace*}

\newcommand{\vp}{\varphi}
\newcommand{\lbal}{\left\{ \begin{array}{l}}
\newcommand{\lball}{\left\{ \begin{array}{ll}}
\newcommand{\ear}{\end{array} \right.}

\newcommand{\abs}{\\[5pt]}

\newcommand{\adb}{\allowdisplaybreaks}

\newcommand{\tme}{T_{max,\eps}}

\newcommand{\ueps}{u_\eps}
\newcommand{\veps}{v_\eps}
\newcommand{\weps}{w_\eps}
\newcommand{\heps}{h_\eps}

\newcommand{\yeps}{y_\eps}
\newcommand{\zeps}{z_\eps}
\newcommand{\Teps}{\Theta_\eps}
\newcommand{\gaeps}{\gamma_\eps}

\newcommand{\wepsx}{w_{\eps x}}
\newcommand{\wepsxx}{w_{\eps xx}}

\newcommand{\wepst}{w_{\eps t}}
\newcommand{\wepsxt}{w_{\eps xt}}

\newcommand{\vepsx}{v_{\eps x}}
\newcommand{\vepsxx}{v_{\eps xx}}
\newcommand{\vepsxxx}{v_{\eps xxx}}
\newcommand{\vepsxxxx}{v_{\eps xxxx}}
\newcommand{\vepst}{v_{\eps t}}
\newcommand{\vepsxt}{v_{\eps xt}}
\newcommand{\vepsxxt}{v_{\eps xxt}}
\newcommand{\uepsx}{u_{\eps x}}
\newcommand{\uepsxx}{u_{\eps xx}}
\newcommand{\uepsxxx}{u_{\eps xxx}}
\newcommand{\uepsxxxx}{u_{\eps xxxx}}

\newcommand{\uepst}{u_{\eps t}}
\newcommand{\uepsxt}{u_{\eps xt}}
\newcommand{\uepsxxt}{u_{\eps xxt}}
\newcommand{\Tepsx}{\Theta_{\eps x}}
\newcommand{\Tepsxx}{\Theta_{\eps xx}}
\newcommand{\Tepst}{\Theta_{\eps t}}

\newcommand{\zepsx}{z_{\eps x}}

\newcommand{\hg}{\wh{\gamma}}
\newcommand{\epss}{\eps_\star}

\newcommand{\ugs}{\gamma_\star}
\newcommand{\ogs}{{\gamma^\star}}
\newcommand{\Ths}{\Theta_\star}
\newcommand{\Aeps}{A_\eps}
\newcommand{\hepsx}{h_{\eps x}}
\newcommand{\hepst}{h_{\eps t}}
\newcommand{\Gas}{{\Gamma^\star}}
\newcommand{\DN}{D_{\mathcal{N}}}
\newcommand{\DD}{D_{\mathcal{D}}}
\newcommand{\ks}{\kappa_{\star}}
\begin{document}
\adb
\title{Global solutions and large time stabilization\\
in a model for thermoacoustics in a standard linear solid}
\author{
Tobias Black\footnote{tblack@math.uni-paderborn.de}\\
{\small Universit\"at Paderborn}\\
{\small  Institut f\"ur Mathematik}\\
{\small 33098 Paderborn, Germany}\\
\and
Michael Winkler\footnote{michael.winkler@math.uni-paderborn.de}\\
{\small Universit\"at Paderborn}\\
{\small Institut f\"ur Mathematik}\\
{\small 33098 Paderborn, Germany}
}
\date{}
\maketitle
\begin{abstract}
\noindent 
This manuscript is concerned with the one-dimensional system 
\begin{align*}
\left\lbrace
\begin{array}{r@{}l}
	\tau u_{ttt} + \al u_{tt} \,&= b \big(\gamma(\Theta) u_{xt}\big)_x + \big( \gamma(\Theta) u_x\big)_x, \\[1mm]
	\Theta_t \,&= D \Theta_{xx} + b\gamma(\Theta) u_{xt}^2,
\end{array}
\right.
\end{align*}
which is connected to the simplified modeling of heat generation in Zener type materials 
subject to stress from acoustic waves. Under the assumption that the coefficients $\tau>0, b>0$ and $\alpha\geq0$ satisfy 
\begin{align}\tag{$\star$}
	\alpha b >\tau,
\end{align}
it is shown that for all $\Theta_\star>0$ one can find $\nu=\nu(D,\tau,\alpha,b,\Theta_\star,\gamma)>0$ such that 
an associated Neumann type
initial-boundary value problem with Neumann data admits a unique time-global solution in a suitable framework of strong solvability whenever the initial temperature distribution fulfills $$\|\Theta_0\|_{L^\infty(\Omega)}\leq \Theta_\star$$ and the derivatives of the initial data are sufficiently small in the sense of
satisfying
$$\int_\Omega u_{0xx}^2 + \int_\Omega (u_{0t})_{xx}^2 + \int_\Omega (u_{0tt})_x^2 < \nu\quad\text{and}\quad
	\|\Theta_{0x}\|_{L^\infty(\Omega)}
	+ \|\Theta_{0xx}\|_{L^\infty(\Omega)}
	< \nu.$$
The constructed solution moreover features an exponential stabilization property for both components. \abs
In particular, the parameter range described by ($\star$)
coincides with the full stability regime known for the corresponding Moore--Gibson--Thompson equation 
despite the fairly strong nonlinear coupling to the temperature variable.\abs
\end{abstract}
%
%
%
%
%
%
%
\newpage
\section{Introduction}\label{intro}
The transmission of acoustic waves through solid materials may lead to a temperature increase due to dissipative mechanical energy loss originating from the small displacements physical particles are subjected to in response to the acoustic disturbance (\cite{Yamaya,Khanghahi-Bala,Guo-Vavilov}).
One classical approach to model the thermo-viscoelastic evolution of a medium is to augment the Kelvin--Voigt equation, which describes the material response to physical stress while accounting for elastic stiffness and viscous damping, via coupling with an equation for the heat conduction according to Fourier's law (\cite{mielke_roubicek}). Essentially, this leads to a system of PDEs with structural similarities to the standard wave-equation for the mechanical displacement and the classical heat equation for the temperature variable. Due to these quite favorable characteristics, these systems have undergone thorough mathematical analysis in the past decades (\cite{dafermos,roubicek,blanchard_guibe,racke_zheng_JDE1997,mielke_roubicek,badal-friedrich2022}).\abs
Recently, approaches going beyond these traits have garnered more interest of the literature. Indeed, assuming the viscoelastic material to be of the more complex Zener type the well-known Moore--Gibson--Thompson equation (\cite{moore_gibson,thompson}) with prototypical form
\begin{align}\label{MGT}
\tau u_{ttt}+\alpha u_{tt}=\widehat{\gamma}\Delta u_t +\gamma\Delta u
\end{align}
lies at the core of the model. Herein, $u=u(x,t)$ denotes the displacement, $\tau$ represents thermal relaxation time, and the coefficients $\gamma,\widehat{\gamma}$ effectively encode stiffness and viscosity of the material. Having a third-order time derivative as conspicuous feature poses a significant mathematical challenge --  even in the setting with linear diffusion and constant coefficients -- but  leads to colorful qualitative solution behavior, ranging from large time decay whenever $\alpha\widehat{\gamma}>\tau\gamma$, to the possibility of solutions with infinite-time grow-up under the condition $\alpha\widehat{\gamma}<\tau\gamma$ (\cite{kal_las_marchand2011,marchand_mcdevitt_triggiani2012,delloro_pata_AMOP2017}). We would also like to refer the reader to  \cite{chen_ikehata,chen2025} for relevant results in whole space settings and to \cite{lasiecka_wang,lasiecka_wang2,delloro_las_pata2016} for integro-differential variants involving memory terms.\abs
Before going into more detail for our current endeavor, let us also briefly mention semilinear variants of \eqref{MGT} with an additional nonlinear term akin to 
\begin{align}\label{JMGT}
\tau u_{ttt}+\alpha u_{tt}=\widehat{\gamma}\Delta u_t +\gamma\Delta u+\big(f(u)\big)_{tt},
\end{align}
 which are called Jordan--Moore--Gibson-Thompson equations. When $f(u)=u^2$, it is known that the equation admits global solutions originating from suitably small initial data when considered in the parametric regime $\alpha\widehat{\gamma}>\tau\gamma$ either in bounded domains (\cite{kaltenbacher_lasiecka_pospieszalska}) or in the whole space (\cite{racke_saidhouari,said_JMGT}), and hence strikingly resembles previous results for \eqref{MGT}. Further results are concerned with finite-time blow-up (\cite{nikolic_win}) and nonlinear terms depending on the solution gradient (\cite{kaltenbacher_nikolic_M3AS2019}).\abs
From a conceptual standpoint, the realizations above feature quite limited influence of the temperature on the displacement, as the constant thermal relaxation is the only ingredient related to the temperature. Experiments, however, suggest that the elastic material parameters $\gamma,\widehat{\gamma}$ may depend on the temperature (see e.g. \cite{friesen}). Accordingly, expecting a more explicit coupling, the temperature could be included as state variable $\Theta$ governed by an appropriate heat equation. Following this approach (refer to \cite{claes_win2,Boley2012} for more details on system derivation), while also assuming that the coefficients $\gamma,\widehat{\gamma}$ are dependent on $\Theta$ and fulfill $\widehat{\gamma}(\Theta)=b\gamma(\Theta)$ for some $b>0$,  gives rise to a system of the form
\bas
	\lbal
	\tau u_{ttt} + \al u_{tt} =  b\big(\gamma(\Theta) u_{xt}\big)_x + \big( \gamma(\Theta) u_x\big)_x, \\[1mm]
	\Theta_t = D \Theta_{xx} + b\gamma(\Theta) u_{xt}^2.
	\ear
\eas
This model archetype was the focal point of inquiry in the recent works \cite{claes_win2,fricke_win}, where local existence and uniqueness of strong solutions and local existence of classical solutions was discussed, respectively, and will be the main object of our study.\abs
Couplings with temperature dependent coefficients have previously been considered in associated second-order problems. Let us briefly mention systems of Westervelt--Pennes type, wherein displacement is modeled by the second-order Westervelt equation featuring a nonlinear as in \eqref{JMGT} and coupling with a temperature state variable satisfying Pennes equation is assumed. For these hyperbolic systems results on global small-data solutions can be found in the recent literature \cite{careaga_nikolic_said_JNLS2025,benabbas_said_SIMA2024}). For further closely related variants we refer the reader to \cite{claes_lankeit_win,fricke,meyer,win_AMOP,win_SIMA} and references therein.\abs
{\bf Main results.} \quad
Motivated by the local existence results in \cite{claes_win2,fricke_win} and the asymptotic stability property in the dissipation-dominated parameter regime of \eqref{MGT}, we are now going to consider global existence and asymptotic stability in the problem
\be{0}
	\lball
	\tau u_{ttt} + \al u_{tt} = b \big(\gamma(\Theta) u_{xt}\big)_x + \big( \gamma(\Theta) u_x\big)_x,
	\qquad & x\in\Om, \ t>0, \\[1mm]
	\Theta_t = D \Theta_{xx} + b\gamma(\Theta) u_{xt}^2,
	& x\in\Om, \ t>0, \\[1mm]
	u_x=0, \quad \Theta_x=0,
	& x\in\pO, \ t>0, \\[1mm]
	u(x,0)=u_0(x), \quad u_t(x,0)=u_{0t}(x), \quad u_{tt}(x,0)=u_{0tt}(x), \quad \Theta(x,0)=\Theta_0(x),
	& x\in\Om.
	\ear
\ee
Our main results assert that despite the fairly strong coupling to the temperature variable appearing therein, 
the nonlinear problem (\ref{0}) inherits from the corresponding
version of (\ref{MGT}) a property of asymptotic stability throughout the full range $\al b>\tau$ within which the latter is dissipation-dominated; in particular, the following statement in this direction does not impose any 
smallness condition on the size of the initial temperature distribution when measured in $L^\infty$.
Here and below, given a bounded interval $\Om\subset\R$ we let 
$W^{2,p}_N(\Om):=\{\vp\in W^{2,p}(\Om) \ | \ \frac{\pa\vp}{\pa\nu}=0 \mbox{ on } \pO\}$ for $p\in (1,\infty]$. 
\begin{theo}\label{theo12}
  Let $\Om\subset\R$ be an open bounded interval, let $D>0$, and let $\tau>0,\al>0$ and $b>0$ be such that
  \be{12.01}
	\al b >\tau.
  \ee
  Then whenever
  \be{gamma}
	\gamma\in C^2([0,\infty))
	\quad \mbox{is such that $\gamma>0$ on $[0,\infty)$,}
  \ee
  given any $\Ths>0$ one can find $\nu=\nu(D,\tau,\al,b,\Ths,\gamma)>0$, $\ks=\ks(D,\tau,\al,b,\Ths,\gamma)>0$ and
  $C=C(D,\tau,\al,b,\Ths,\gamma)>0$ with the property that if
  \be{init} 
	\lbal
	u_0\in W^{2,2}_N(\Om)
	\mbox{ is such that } \io u_0=0, \\[1mm]
	u_{0t}\in W^{2,2}_N(\Om)
	\mbox{ is such that } \io u_{0t}=0, \\[1mm]
	u_{0tt} \in W^{1,2}(\Om)
	\mbox{ is such that } \io u_{0tt}=0
	\qquad \mbox{and} \\[1mm]
	\Theta_0 \in W^{2,\infty}_N(\Om) 
	\mbox{ is such that $\Theta_0\ge 0$ in $\Om$,}
	\ear
  \ee
  and if moreover
  \be{12.1}
	\io u_{0xx}^2 + \io (u_{0t})_{xx}^2 + \io (u_{0tt})_x^2 < \nu
  \ee
  as well as
  \be{12.2}
	\|\Theta_0\|_{L^\infty(\Om)} < \Ths
	\qquad \mbox{and} \qquad
	\|\Theta_{0x}\|_{L^\infty(\Om)}
	+ \|\Theta_{0xx}\|_{L^\infty(\Om)}
	< \nu,
  \ee
  then (\ref{0}) admits a uniquely determined
  strong solution $(u,\Theta)$, in the sense of Definition \ref{dw} below, which satisfies
  \be{12.3}
	\io u_{xx}^2(\cdot,t) + \io u_{xxt}^2(\cdot,t) + \io u_{xtt}^2(\cdot,t) \le C e^{-\ks t}
	\qquad \mbox{for a.e.~} t>0,
  \ee
  and which is such that there exists a number $\Theta_\infty=\Theta_\infty(D,\tau,\al,b,u_0,u_{0t},u_{0tt},\Theta_0)>0$ fulfilling
  \be{12.4}
	\|\Theta(\cdot,t)-\Theta_\infty\|_{L^\infty(\Om)} \le C e^{-\ks t}
	\qquad \mbox{for all } t>0.
  \ee
\end{theo}
{\bf Remark.} \quad
  The assumptions $\io u_0=\io u_{0t}=\io u_{0tt}=0$ in (\ref{init}) are imposed here exclusively for convenience in presentation.
  In fact, given $\Ths>0$ and arbitrary initial data 
  $(u_0,u_{0t},u_{0tt},\Theta_0) \in W^{2,2}_N(\Om) \cap W^{2,2}_N(\Om) \cap W^{1,2}(\Om) \cap W^{2,\infty}_N(\Om)$
  merely satisfying $\Theta_0\ge 0$ as well as (\ref{12.1}) and (\ref{12.2}) with $\nu=\nu(D,\tau,\al,b,\Ths,\gamma)$ as above,
  one can readily verify that if we let $y_0:=\frac{1}{|\Om|} \io u_0$, 
  $y_1:=\frac{1}{|\Om|} \io u_{0t}$, $y_2:=\frac{1}{|\Om|} \io u_{0tt}$
  and $(\wt{u}_0,\wt{u}_{0t},\wt{u}_{0tt},\wt{\Theta}_0)(x):=(u_0(x)-y_0,u_{0t}(x)-y_1,u_{0tt}(x)-y_2,\Theta_0(x))$ for $x\in\Om$, 
  and let 
  $(\wt{u},\wt{\Theta})$ denote the solution of (\ref{0}) emanating from the initial data 
  $(\wt{u}_0,\wt{u}_{0t},\wt{u}_{0tt},\wt{\Theta}_0)$, then writing
  $u(x,t):=\wt{u}(x,t)+y(t)$ and $\Theta(x,t):=\wt{\Theta}(x,t)$ for $(x,t)\in\Om\times (0,\infty)$, with $y\in C^3([0,\infty))$
  denoting the solution of $y'''-\al y''=0$ on $(0,\infty)$ with $y(0)=y_0, y'(0)=y_1$ and $y''(0)=y_2$, then
  $(u,\Theta)$ forms a global strong solution of (\ref{0}) satisfying (\ref{12.3}) and (\ref{12.4}).\abs
{\bf Main ideas.} \quad
Starting with a classical solution to a diffusion-perturbed regularization of \eqref{0} -- which can be obtained by utilizing well-established fixed point arguments --, the first goal of our inquest is to ensure that the approximate solutions are indeed time-global solutions under the assumptions on the initial data as presented in \eqref{12.1} and \eqref{12.2}. To this end, we will invoke a self-map argument (see Lemma~\ref{lem10}) wherein exponential decay properties for the third- and second-order quantities $\|u_{xtt}\|_{L^2(\Om)}$, $\|u_{xxt}\|_{L^2(\Om)}$ and $\|u_{xx}\|_{L^2(\Om)}$ (cf. Corollary~\ref{cor7}) play an essential role. In pursuit of these estimates, our approach is intricately arranged around the analysis of evolution properties enjoyed by functionals of the form 
\bea{en}
	y(t) 
	&:=& \frac{\tau}{2} \io u_{xtt}^2
	+ \frac{b}{2} \io \gamma(\Theta) u_{xxt}^2
	+ \frac{B+b\del}{2} \io \gamma(\Theta) u_{xx}^2
	+ \frac{\al B - \tau\del}{2} \io u_{xt}^2 \nn\\
	& & + \io \gamma'(\Theta) u_{xx} u_{xxt}
	+ \tau B \io u_{xt} u_{xtt}
	+ \tau\del \io u_x u_{xtt}
	+ \al\del \io u_x u_{xt},
\eea
(see Lemma~\ref{lem4}) with suitably chosen $B=B(\tau,\al,b)>0$ and $\del=\del(\tau,\al,b,\Ths,\gamma)$ ensuring that $y$ fulfills a genuine Lyapunov-type differential inequality (Lemma~\ref{lem6}) that entails exponential decay estimates. Since the previously derived bounds are uniform in nature, we can finally turn to an Aubin--Lions type argument to construct a limit solution (see Section~\ref{sec7}).
\mysection{Solution concept and reduction to a family of diffusion-perturbed regularized first-order systems}
The following natural concept of solvability essentially coincides with that introduced in \cite[Definition 2.1]{claes_win2}.
\begin{defi}\label{dw}
  Let $D>0$, $\tau>0$, $\al\ge 0$ and $b>0$, let
  $\gamma\in C^0([0,\infty))$ nonnegative, 
  let
  $u_0\in W^{2,2}_N(\Om), u_{0t}\in W^{2,2}_N(\Om)$, $u_{0tt} \in W^{1,2}(\Om)$ and $\Theta_0\in C^0(\bom)$
  be such that $\Theta_0\ge 0$, and suppose that
  \be{w1}
	\lbal
	u \in C^0([0,\infty);C^1(\bom)) \cap L^\infty_{loc}([0,\infty);W^{2,2}_N(\Om))
	\qquad \mbox{and} \\[1mm]
	\Theta\in C^0(\bom\times [0,\infty)) \cap C^{2,1}(\bom\times (0,\infty))
	\ear
  \ee
  be such that $\Theta\ge 0$.
  Then we will say that $(u,\Theta)$ is a {\em strong solution} of (\ref{0}) if
  \be{w2}
	\lbal
	u_t \in C^0([0,\infty);C^1(\bom)) \cap L^\infty_{loc}([0,\infty);W^{2,2}_N(\Om)), \\[1mm]
	u_{tt} \in C^0(\bom\times [0,\infty)) \cap L^\infty_{loc}([0,\infty);W^{1,2}(\Om)), \\[1mm]
	\Theta_x \in L^\infty_{loc}(\bom\times [0,\infty))
	\qquad \mbox{and} \\[1mm]
	\Theta_t \in L^\infty_{loc}(\bom\times [0,\infty)),
	\ear
  \ee
  if 
  \be{w3}
	u(\cdot,0)=u_0,
	\quad u_t(\cdot,t)=u_{0t}\quad \mbox{and} \quad
	\Theta(\cdot,0)=\Theta_{0}
	\quad \mbox{in } \Om,
  \ee  
  if 
  \bea{wu}
	- \tau \int_0^\infty\! \io u_{tt} \vp_t 
	- \tau \io u_{0tt} \vp(\cdot,0)
	+ \al \int_0^\infty\! \io u_{tt} \vp
	= - b \int_0^\infty\! \io \gamma(\Theta) u_{xt} \vp_x
	- \int_0^\infty\! \io \hg(\Theta) u_x\vp_x
  \eea
  for each $\vp\in C_0^\infty(\bom\times [0,\infty))$, and if 
  $\Theta_t=D\Theta_{xx} +  b\gamma(\Theta) u_{xt}^2$ in $\Om\times (0,\infty)$ as well as $\Theta_x=0$ on $\pO\times (0,\infty)$.
\end{defi}
In order to construct such a solution under conditions compatible with the statement of Theorem \ref{theo12},
assuming (\ref{gamma}) and (\ref{init}) we fix
\be{gaeps0}
	(\gaeps)_{\eps\in (0,1)} \subset C^\infty([0,\infty);[0,\infty))
\ee
as well as
\be{ie0}
	(u_{0\eps})_{\eps\in (0,1)} \subset C^\infty(\bom),
	\quad 
	(v_{0\eps})_{\eps\in (0,1)} \subset C^\infty(\bom),
	\quad
	(w_{0\eps})_{\eps\in (0,1)} \subset C^\infty(\bom)
	\quad \mbox{and} \quad
	(\Theta_{0\eps})_{\eps\in (0,1)} \subset C^\infty(\bom)
\ee
such that
\be{gaepsc}
	\gaeps \to \gamma
	\mbox{ in } C^2_{loc}([0,\infty))
	\qquad \mbox{as } \eps\searrow 0,
\ee
that all the functions $u_{0\eps x}$, $v_{0\eps x}$, $w_{0\eps x}$ and $\Theta_{0\eps x}$ are compactly supported in $\Om$ with
\bas
	\io u_{0\eps} = 0,\quad \io v_{0\eps} = 0
	\quad \mbox{and} \quad \io w_{0\eps}=0
\eas
as well as $\Theta_{0\eps} \ge 0$ in $\Om$ and 
\be{ie1}
	\|\Theta_{0\eps x}\|_{L^\infty(\Om)}+ \|\Theta_{0\eps xx}\|_{L^\infty(\Om)} 
	\le \|\Theta_{0 x}\|_{L^\infty(\Om)}+\|\Theta_{0 xx}\|_{L^\infty(\Om)} + \eps
\ee
for all $\eps\in (0,1)$,
and that
\be{iec}
	\lbal
	u_{0\eps} \to u_0
	\quad \mbox{ in } W^{2,2}(\Om),
	\\[1mm]
	v_{0\eps} \to u_{0t}
	\quad \mbox{ in } W^{2,2}(\Om),
	\\[1mm]
	w_{0\eps} \to u_{0tt}
	\quad \mbox{ in } W^{1,2}(\Om)
	\qquad \mbox{and} 
	\\[1mm]
	\Theta_{0\eps} \to \Theta_0
	\quad \mbox{ in } L^\infty(\Om)
	\ear
\ee
as well as
\be{iec2}
	\sqrt{\eps} u_{0\eps xxx} \to 0
	\quad \mbox{in } L^2(\Om)
\ee
as $\eps\searrow 0$.
Then, as noted in \cite[Lemma 3.1]{claes_win2},
for any choice of $\eps\in (0,1)$ the problem
\be{0eps}
	\lball
	\tau \wepst = \eps \wepsxx + b \big(\gaeps(\Teps) \vepsx\big)_x + \big(\gaeps(\Teps)\uepsx\big)_x - \al\weps,
	\quad & x\in\Om, \ t>0, \\[1mm]
	\vepst = \eps \vepsxx + \weps,
	\quad & x\in\Om, \ t>0, \\[1mm]
	\uepst = \eps \uepsxx + \veps,
	\quad & x\in\Om, \ t>0, \\[1mm]
	\Tepst = D \Tepsxx + b \gaeps(\Teps) \vepsx^2,
	\quad & x\in\Om, \ t>0, \\[1mm]
	\wepsx=\vepsx=\uepsx=\Tepsx=0,
	\quad & x\in\pO, \ t>0, \\[1mm]
	\weps(x,0)=w_{0\eps}(x), \ \; \veps(x,0)=v_{0\eps}(x), \ \; \ueps(x,0)=u_{0\eps}(x), \ \; \Teps(x,0)=\Theta_{0\eps}(x),
	\quad & x\in\Om,
	\ear
\ee
admits a local classical solution in the following sense:
\begin{lem}\label{lem_loc}
  Let $D>0, \tau>0, \al>0$ and $b>0$, and assume (\ref{gaeps0}) and (\ref{ie0}).
  Then for each $\eps\in (0,1)$, there exist $\tme\in (0,\infty]$ and 
  \be{reg}
	\lbal
	\weps \in C^{2,1}(\bom\times [0,\tme)) \cap C^\infty(\bom\times (0,\tme)), \\[1mm]
	\veps \in C^{2,1}(\bom\times [0,\tme)) \cap C^\infty(\bom\times (0,\tme)), \\[1mm]
	\ueps \in C^{2,1}(\bom\times [0,\tme)) \cap C^\infty(\bom\times (0,\tme))
	\qquad \mbox{and} \\[1mm]
	\Teps \in C^{2,1}(\bom\times [0,\tme)) \cap C^\infty(\bom\times (0,\tme))
	\ear
  \ee
  such that $\Teps\ge 0$ in $\Om\times (0,\tme)$, that $(\weps,\veps,\ueps,\Teps)$ solves 
  (\ref{0eps}) classically in $\Om\times (0,\tme)$, and that
  \bea{ext_eps}
	& & \hs{-15mm}
	\mbox{if $\tme<\infty$, \quad then \quad} \nn\\
	& & \hs{-8mm}
	\limsup_{t\nearrow\tme} \Big\{ 
	\|\weps(\cdot,t)\|_{L^\infty(\Om)} 
	+ \|\veps(\cdot,t)\|_{W^{1,2}(\Om)} 
	+ \|\ueps(\cdot,t)\|_{W^{1,2}(\Om)} 
	+ \|\Teps(\cdot,t)\|_{L^\infty(\Om)} 
	\Big\}
	= \infty.
  \eea
  Moreover,
  \be{mass}
	\io \weps(\cdot,t) = \io \veps(\cdot,t) = \io \ueps(\cdot,t) =0
	\qquad \mbox{for all } t\in (0,\tme).
  \ee
\end{lem}
Throughout the sequel, we let the open bounded interval $\Om\subset\R$ be fixed, and in order to unambiguously trace dependencies 
of the estimates to be subsequently derived, in several places below we only rely on certain parts of the assumptions
in (\ref{gamma}), (\ref{init}), (\ref{gaeps0}), (\ref{ie0}), (\ref{gaepsc}), (\ref{ie1}), (\ref{iec}) and (\ref{iec2})

\mysection{Elementary testing procedures}
Let us begin our construction of an energy functional by recording the following outcome of a straightforward
testing procedure, detailed in \cite[Lemma 3.5]{claes_win2} for a close relative of (\ref{0eps}).
\begin{lem}\label{lem1}
  Let $D>0, \tau>0, \al>0$ and $b>0$, assume (\ref{gaeps0}) and (\ref{ie0}), and let $\eps\in (0,1)$. 
  Then
  \bea{1.1}
	& & \hs{-20mm}
	\frac{d}{dt} \bigg\{
	\frac{\tau}{2} \io \wepsx^2
	+ \frac{b}{2} \io \gaeps(\Teps) \vepsxx^2
	+ \io \gaeps(\Teps) \uepsxx \vepsxx
	+ \eps \io \gaeps(\Teps) \uepsxxx^2 \bigg\} \nn\\
	& & + \al \io \wepsx^2
	+ \eps \io \wepsxx^2
	+ b \eps \io \gaeps(\Teps) \vepsxxx^2
	+ 2\eps^2 \io \gaeps(\Teps) \uepsxxxx^2 \nn\\
	&=& \io \gaeps(\Teps) \vepsxx^2
	+ \frac{b}{2} \io \gaeps'(\Teps) \Tepst \vepsxx^2
	+ \io \gaeps'(\Teps) \Tepst \uepsxx \vepsxx \nn\\
	& & + b \io \gaeps'(\Teps) \Tepsx \vepsxx \wepsx
	+ \io \gaeps'(\Teps) \Tepsx \uepsxx \wepsx 
	+ b \io \gaeps'(\Teps) \Tepsxx \vepsx \wepsx \nn\\
	& & + \io \gaeps'(\Teps) \Tepsxx \uepsx \wepsx
	+ b \io \gaeps''(\Teps) \Tepsx^2 \vepsx \wepsx
	+ \io \gaeps''(\Teps) \Tepsx^2 \uepsx \wepsx \nn\\
	& & + \eps \io \gaeps'(\Teps) \Tepst \uepsxxx^2
	- b \eps \io \gaeps'(\Teps) \Tepsx \vepsxx \vepsxxx \nn\\
	& & - \eps \io \gaeps'(\Teps) \Tepsx \uepsxx \vepsxxx
	- \eps \io \gaeps'(\Teps) \Tepsx \uepsxxx \vepsxx \nn\\
	& & - 2\eps^2 \io \gaeps'(\Teps) \Tepsx \uepsxxx \uepsxxxx
	\qquad \mbox{for all } t\in (0,\tme).
  \eea
\end{lem}
The identity (\ref{1.1}) will be supplemented by an evolution property of a second family
of functionals, the introduction of which can be motivated by the ambition to adequately compensate the first
summand on the right of (\ref{1.1}):
\begin{lem}\label{lem2}
  If $D>0, \tau>0, \al>0$ and $b>0$, and if (\ref{gaeps0}) and (\ref{ie0}) hold, then 
  \bea{2.1}
	& & \hs{-20mm}
	\frac{d}{dt} \bigg\{
	\tau \io \vepsx\wepsx
	+ \frac{\al}{2} \io \vepsx^2
	+ \frac{1}{2} \io \gaeps(\Teps) \uepsxx^2
	+ \frac{(1+\tau)\eps}{2} \io \vepsxx^2 \bigg\} \nn\\
	& & + b \io \gaeps(\Teps) \vepsxx^2
	+ \al \eps \io \vepsxx^2
	+ \eps \io \gaeps(\Teps)\uepsxxx^2
	+ (1+\tau) \eps^2 \io \vepsxxx^2 \nn\\
	&=& \tau \io \wepsx^2
	+ \frac{1}{2} \io \gaeps'(\Teps) \Tepst \uepsxx^2
	- b \io \gaeps'(\Teps) \Tepsx \vepsx \vepsxx
	- \io \gaeps'(\Teps) \Tepsx \uepsx\vepsxx \nn\\
	& & - \eps \io \gaeps'(\Teps) \Tepsx \uepsxx \uepsxxx
	\qquad \mbox{for all $t\in (0,\tme)$ and } \eps\in (0,1).
  \eea
\end{lem}
\proof
  Fixing $\eps\in (0,1)$, we multiply the first equation in (\ref{0eps}) by $-\vepsxx$ and integrate by parts to see that
  for all $t\in (0,\tme)$,
  \bea{2.2}
	\tau \io \vepsx \wepsxt
	+ \al \io \vepsx \wepsx
	&=& - \eps\io \vepsxx \wepsxx
	- b \io \gaeps(\Teps) \vepsxx^2
	- b \io \gaeps'(\Teps)\Tepsx \vepsx\vepsxx \nn\\
	& & - \io \gaeps(\Teps) \uepsxx \vepsxx
	- \io \gaeps'(\Teps) \Tepsx \uepsx \vepsxx,
  \eea
  and here the identities $\vepsxt=\eps\vepsxxx+\wepsx$ and $\vepsxxt = \eps \vepsxxxx+\wepsxx$ imply that
  for all $t\in (0,\tme)$,
  \bea{2.3}
	\tau \io \vepsx \wepsxt
	&=& \tau\frac{d}{dt} \io \vepsx \wepsx
	- \tau \io \big\{ \eps \vepsxxx+\wepsx\big\}\cdot \wepsx \nn\\
	&=& \tau\frac{d}{dt} \io \vepsx \wepsx
	+ \tau \eps \io \vepsxx \wepsxx
	- \tau \io \wepsx^2 \nn\\
	&=& \tau\frac{d}{dt} \io \vepsx \wepsx
	+ \tau \eps \io \vepsxx \cdot \big\{ \vepsxxt - \eps \vepsxxxx\big\}
	- \tau \io \wepsx^2 \nn\\
	&=& \tau\frac{d}{dt} \io \vepsx \wepsx
	+ \frac{\tau \eps}{2} \frac{d}{dt} \io \vepsxx^2
	+ \tau\eps^2 \io \vepsxxx^2
	- \tau \io \wepsx^2,
  \eea
  because clearly $\vepsxxx=0$ on $\pO\times (0,\tme)$.
  Similarly,
  \bea{2.4}
	\al \io \vepsx\wepsx
	&=& \al \io \vepsx\cdot\big\{ \vepsxt - \eps \vepsxxx\big\} \nn\\
	&=& \frac{\al}{2} \frac{d}{dt} \io \vepsx^2
	+ \al\eps \io \vepsxx^2
  \eea
  and
  \bea{2.5}
	- \eps \io \vepsxx\wepsxx
	&=& - \eps \io \vepsxx\cdot\big\{ \vepsxxt - \eps\vepsxxxx\big\} \nn\\
	&=& - \frac{\eps}{2} \frac{d}{dt} \io \vepsxx^2
	- \eps^2 \io \vepsxxx^2
  \eea
  for all $t\in (0,\tme)$, while using the third equation in (\ref{0eps}) we find that
  \bea{2.6}
	- \io \gaeps(\Teps) \uepsxx\vepsxx
	&=& - \io \gaeps(\Teps) \uepsxx\cdot\big\{ \uepsxxt - \eps\uepsxxxx\big\} \nn\\
	&=& - \frac{1}{2} \frac{d}{dt} \io \gaeps(\Teps)\uepsxx^2
	+ \frac{1}{2} \io \gaeps(\Teps) \Tepst \uepsxx^2 \nn\\
	& & - \eps \io \gaeps(\Teps) \uepsxxx^2
	- \eps \io \gaeps'(\Teps) \Tepsx \uepsxx \uepsxxx
  \eea
  for all $t\in (0,\tme)$.
  Inserting (\ref{2.3})-(\ref{2.6}) into (\ref{2.2}) confirms (\ref{2.1}).
\qed
Expressions containing $\uepsxx^2$, finally, can be controlled by the dissipated quantity appearing in the course
of a third testing procedure:
\begin{lem}\label{lem3}
  Assuming that $D>0, \tau>0, \al>0$ and $b>0$, and that (\ref{gaeps0}) and (\ref{ie0}) hold, for any $\eps\in (0,1)$
  and $t\in (0,\tme)$ we have
  \bea{3.1}
	& & \hs{-14mm}
	\frac{d}{dt} \bigg\{
	\tau \io \uepsx\wepsx
	- \frac{\tau}{2} \io \vepsx^2
	+ \al \io \uepsx\vepsx
	+ \frac{b}{2} \io \gaeps(\Teps) \uepsxx^2 \bigg\}
	+ \io \gaeps(\Teps) \uepsxx^2
	+ b\eps \io \gaeps(\Teps) \uepsxxx^2 \nn\\
	&=& \al \io \vepsx^2 
	+ \frac{b}{2} \io \gaeps'(\Teps)\Tepst \uepsxx^2
	- b \io \gaeps'(\Teps)\Tepsx \vepsx \uepsxx
	- \io \gaeps'(\Teps)\Tepsx \uepsx \uepsxx \nn\\
	& & - (1+\tau)\eps \io \uepsxx \wepsxx
	- 2\al\eps \io \uepsxx \vepsxx
	+ \tau\eps \io \vepsxx^2 
	- b\eps \io \gaeps'(\Teps) \Tepsx \uepsxx \uepsxxx.
  \eea
\end{lem}
\proof
  Given $\eps\in (0,1)$, for $t\in (0,\tme)$ we test the first equation in (\ref{0eps}) by $-\uepsxx$ to obtain the identity
  \bea{3.2}
	\tau \io \uepsx\wepsxt
	+ \al \io \uepsx\wepsx
	&=& - \eps \io \uepsxx\wepsxx 
	- b\io \gaeps(\Teps) \uepsxx \vepsxx
	- b\io \gaeps'(\Teps)\Tepsx \vepsx\uepsxx \nn\\
	& & - \io \gaeps(\Teps) \uepsxx^2
	- \io \gaeps'(\Teps)\Tepsx \uepsx \uepsxx,
  \eea
  in which using that $\uepsxt=\eps\uepsxxx + \vepsx$ and $\wepsx=\vepsxt - \eps\vepsxxx$ we see that
  \bea{3.3}
	\tau \io \uepsx\wepsxt 
	&=& \tau\frac{d}{dt} \io \uepsx\wepsx
	- \tau \io \uepsxt \wepsx \nn\\
	&=& \tau\frac{d}{dt} \io \uepsx\wepsx
	- \tau \eps \io \uepsxxx\wepsx
	- \tau \io \vepsx \wepsx \nn\\
	&=& \tau\frac{d}{dt} \io \uepsx\wepsx
	+ \tau \eps \io \uepsxx\wepsxx
	- \tau \io \vepsx \vepsxt 
	+ \tau\eps \io \vepsx \vepsxxx \nn\\
	&=& \tau\frac{d}{dt} \io \uepsx\wepsx
	+ \tau \eps \io \uepsxx\wepsxx 
	- \frac{\tau}{2} \frac{d}{dt} \io \vepsx^2
	- \tau\eps \io \vepsxx^2
  \eea
  for all $t\in (0,\tme)$.
  Likweise,
  \bea{3.4}
	\al \io \uepsx\wepsx
	&=& \al \io \uepsx \vepsxt
	- \al\eps \io \uepsx \vepsxxx \nn\\
	&=& \al\frac{d}{dt} \io \uepsx\vepsx
	- \al \io \uepsxt \vepsx
	+ \al\eps \io \uepsxx\vepsxx \nn\\
	&=& \al\frac{d}{dt} \io \uepsx\vepsx
	- \al \eps \io \uepsxxx \vepsx
	- \al \io \vepsx^2 
	+ \al\eps \io \uepsxx\vepsxx \nn\\
	&=& \al\frac{d}{dt} \io \uepsx\vepsx
	- \al \io \vepsx^2 
	+ 2 \al\eps \io \uepsxx\vepsxx 
	\qquad \mbox{for all } t\in (0,\tme),
  \eea
  while on the right-hand side of (\ref{3.2}),
  \bea{3.5}
	- b \io \gaeps(\Teps) \uepsxx \vepsxx
	&=& - b \io \gaeps(\Teps) \uepsxx\uepsxxt
	+ b\eps\io \gaeps(\Teps) \uepsxx \uepsxxxx \nn\\
	&=& - \frac{b}{2} \frac{d}{dt} \io \gaeps(\Teps) \uepsxx^2
	+ \frac{b}{2} \io \gaeps'(\Teps) \Tepst \uepsxx^2 \nn\\
	& & - b\eps \io \gaeps(\Teps) \uepsxxx^2
	- b\eps \io \gaeps'(\Teps) \Tepsx \uepsxx \uepsxxx
  \eea
  for all $t\in (0,\tme)$.
  Combining (\ref{3.3})-(\ref{3.5}) with (\ref{3.2}) yields (\ref{3.1}).
\qed
\mysection{A conditional energy property of the mechanical part}
By combining the results from the previous section, we can now characterize the evolution properties of a two-parameter 
family of candidates for an energy functional,
constructed by suitably adapting the expression in (\ref{en}) to the regularized framework of (\ref{0eps}):
\begin{lem}\label{lem4}
  Suppose that $D>0, \tau>0, \al>0$ and $b>0$, assume (\ref{gaeps0}) and (\ref{ie0}), 
  and for $B>0$, $\del>0$ and $\eps\in (0,1)$, let
  \bea{y}
	\yeps(t) \equiv \yeps^{(B,\del)}(t)
	&:=& \frac{\tau}{2} \io \wepsx^2
	+ \frac{b}{2} \io \gaeps(\Teps) \vepsxx^2
	+ \frac{B+b\del}{2} \io \gaeps(\Teps) \uepsxx^2
	+ \frac{\al B - \tau\del}{2} \io \vepsx^2 \nn\\
	& & + \io \gaeps(\Teps) \uepsxx\vepsxx
	+ \eps \io \gaeps(\Teps) \uepsxxx^2
	+ \frac{(1+\tau)B\eps}{2} \io \vepsxx^2 \nn\\
	& & + \tau B \io \vepsx\wepsx
	+ \tau\del \io \uepsx\wepsx
	+ \al\del \io \uepsx\vepsx,
	\qquad t\in [0,\tme).
  \eea
  Then
  \bea{4.1}
	& & \hs{-20mm}
	\yeps'(t)
	+ (\al-\tau B) \io \wepsx^2
	+ (bB-1) \io \gaeps(\Teps) \vepsxx^2
	+ \del \io \gaeps(\Teps) \uepsxx^2 \nn\\
	& & + \eps \io \wepsxx^2
	+ b\eps \io \gaeps(\Teps) \vepsxxx^2
	+ (1+\tau) B\eps^2 \io \vepsxxx^2
	+ \al B\eps \io \vepsxx^2 \nn\\
	& & + (B+b\del) \eps \io \gaeps(\Teps) \uepsxxx^2
	+ 2\eps^2 \io \gaeps(\Teps) \uepsxxxx^2 \nn\\
	&=& \al\del \io \vepsx^2 \nn\\
	& & + \frac{b}{2} \io \gaeps'(\Teps) \Tepst \vepsxx^2
	+ \io \gaeps'(\Teps) \Tepst \uepsxx \vepsxx
	+ \frac{B+b\del}{2} \io \gaeps'(\Teps) \Tepst \uepsxx^2 \nn\\
	& & + b\io \gaeps'(\Teps) \Tepsx \vepsxx \wepsx
	+ \io \gaeps'(\Teps) \Tepsx \uepsxx\wepsx \nn\\
	& & + b\io \gaeps'(\Teps)\Tepsxx \vepsx \wepsx
	+ \io \gaeps'(\Teps) \Tepsxx \uepsx \wepsx \nn\\
	& & + b\io \gaeps''(\Teps) \Tepsx^2 \vepsx\wepsx
	+ \io \gaeps''(\Teps) \Tepsx^2 \uepsx\wepsx \nn\\
	& & - bB \io \gaeps'(\Teps) \Tepsx \vepsx \vepsxx
	- B \io \gaeps'(\Teps)\Tepsx \uepsx\vepsxx \nn\\
	& & - b\del \io \gaeps'(\Teps)\Tepsx \vepsx \uepsxx
	- \del \io \gaeps'(\Teps)\Tepsx\uepsx\uepsxx \nn\\
	& & + \eps \io \gaeps'(\Teps) \Tepst \uepsxxx^2 
	- b\eps \io \gaeps'(\Teps) \Tepsx \vepsxx\vepsxxx
	- \eps \io \gaeps'(\Teps)\Tepsx \uepsxx \vepsxxx \nn\\
	& & - \eps \io \gaeps'(\Teps) \Tepsx \uepsxxx \vepsxx 
	-(B+b\del)\eps \io \gaeps'(\Teps) \Tepsx \uepsxx\uepsxxx
	\nn\\
	& & - 2\eps^2 \io \gaeps'(\Teps) \Tepsx \uepsxxx\uepsxxxx 
	- (1+\tau)\del\eps \io \uepsxx\wepsxx
	- 2\al\del\eps \io \uepsxx\vepsxx \nn\\
	& & + \tau\del\eps \io \vepsxx^2
	\qquad \mbox{for all $t\in (0,\tme)$.}
  \eea
\end{lem}
\proof
  This immediately results upon combining Lemma \ref{lem1} with Lemma \ref{lem2} and Lemma \ref{lem3}.
\qed
Now the stability condition in (\ref{12.01}) enters, ensuring that for some appropriately chosen $B>0$, and any suitably
small $\del>0$, the functional in (\ref{y}) admits a two-sided estimate as long as the behavior of $\gaeps(\Teps)$
can suitably be controlled:
\begin{lem}\label{lem5}
  Assume that $\tau>0$, $\al>0$ and $b>0$ satisfy (\ref{12.01}), 
  let $B>0$ be such that
  \be{5.1}
	\frac{1}{b} < B < \frac{\al}{\tau},
  \ee
  and suppose that $\ugs>0$ and $\ogs>\ugs$. 
  Then there exist $\del_1=\del_1(\tau,\al,b,B,\ugs,\ogs)>0$ as well as $k_i=k_i(\tau,\al,b,B,\ugs,\ogs)>0$, $i\in\{1,2\}$, with the 
  property that if $\del\in (0,\del_1]$ and $D>0$, and if (\ref{gaeps0}) and (\ref{ie0}) hold and $\eps\in (0,1)$ 
  as well as $T\in (0,\tme)$ are such that
  \be{5.2}
	\ugs \le \gaeps(\Teps) \le \ogs
	\qquad \mbox{in } \Om\times (0,T),
  \ee
  then
  \bea{5.3}
	& & \hs{-40mm}
	k_1 \io \wepsx^2 + k_1 \io \vepsxx^2 + k_1 \io \uepsxx^2 + k_1 \eps \io \uepsxxx^2 \nn\\
	&\le& \yeps^{(B,\del)}(t) \nn\\
	&\le& k_2 \io \wepsx^2
	+ k_2 \io \vepsxx^2
	+ k_2 \io \uepsxx^2
	+ k_2\eps \io \uepsxxx^2
  \eea
  for all $t\in (0,T)$.
\end{lem}
\proof
  Let us first recall that a Poincar\'e inequality states that
  \be{5.33}
	\io \vp^2 \le c_1 \io \vp_x^2
	\qquad \mbox{for all } \vp\in W_0^{1,2}(\Om)
  \ee
  with $c_1:=\frac{|\Om|^2}{\pi^2}$.
  Apart from that, relying on the strictness of both inequalities in (\ref{5.1}) we pick $b_1=b_1(\tau,\al,b,B)\in (0,b)$
  and $B_1=B_1(\tau,\al,b,B)\in (0,B)$ such that
  \be{5.4}
	\frac{1}{b_1} < B
	\qquad \mbox{and} \qquad
	\frac{B}{B_1} \cdot B < \frac{\al}{\tau},
  \ee
  where the latter particularly ensures that
  \be{5.50}
	c_2\equiv c_2(\tau,\al,b,B):=\frac{\tau}{2} \cdot \Big(1-\frac{\tau B^2}{\al B_1}\Big)
  \ee
  is positive.
  Given $\ugs>0$ and $\ogs>\ugs$, we can therefore choose $\del_1=\del_1(\tau,\al,b,B,\ugs,\ogs)>0$ small enough fulfilling
  \be{5.51}
	\tau\del \le \frac{\al(B-B_1)}{2}
  \ee
  and
  \be{5.52}
	\Big(\frac{\tau^2}{2c_2} + \frac{\al}{B-B_1} \Big) \cdot \del^2 \cdot c_1
	\le \frac{b\ugs}{2} \cdot \del
  \ee
  for all $\del\in (0,\del_1]$.
  We then let 
  \be{5.6}
	k_1\equiv k_1(\tau,\al,b,B,\ugs,\ogs)
	:=\min \bigg\{ \frac{c_2}{2} \, , \, \frac{(b-b_1)\ugs}{2} \, , \, \frac{(Bb_1-1)\ugs}{2b_1} \, , \, \ugs \bigg\}
  \ee
  as well as
  \bea{5.65}
	k_2\equiv k_2(\tau,\al,b,B,\ugs,\ogs)
	&:=& \max \bigg\{ \frac{\tau}{2} + \frac{\tau^2 B^2}{2\al B_1} + \frac{c_2}{2} \, , \, 
	\frac{(b+b_1)\ogs}{2} + \frac{(1+\tau)B}{2} + \frac{c_1\al (3B+B_1)}{4} \, , \,  \nn\\
	& & \hs{12mm}
	\frac{(B+b\del)\ogs}{2} + \frac{\ogs}{2b_1} + \frac{b\del\ugs}{2} \, , \, 
	\ogs \bigg\},
  \eea
  noting that $k_1$ is positive thanks to the first inequality in (\ref{5.4}).\abs
  Now assuming that $\del\in (0,\del_1]$, and that (\ref{5.2}) be valid with some $\eps\in (0,1)$ and $T\in (0,\tme)$,
  in (\ref{y}) we use Young's inequality to estimate
  \be{5.66}
	\bigg| \io \gaeps(\Teps) \uepsxx\vepsxx\bigg|
	\le \frac{b_1}{2} \io \gaeps(\Teps) \vepsxx^2
	+ \frac{1}{2b_1} \io \gaeps(\Teps) \uepsxx^2
	\qquad \mbox{for all } t\in (0,T),
  \ee
  so that since $\frac{b_1}{2} < \frac{b}{2}$ and $\frac{1}{2b_1} < \frac{B}{2} \le \frac{B+b\del}{2}$, (\ref{5.2}) applies
  so as to warrant that
  \bea{5.7}
	& & \hs{-20mm}
	\frac{b}{2} \io \gaeps(\Teps) \vepsxx^2
	+ \frac{B+b\del}{2} \io \gaeps(\Teps) \uepsxx^2
	+ \io \gaeps(\Teps) \uepsxx\vepsxx \nn\\
	&\ge& \Big(\frac{b}{2}-\frac{b_1}{2}\Big) \io \gaeps(\Teps) \vepsxx^2
	+ \Big(\frac{B+b\del}{2}-\frac{1}{2b_1}\Big) \io \gaeps(\Teps)\uepsxx^2 \nn\\
	&\ge& \frac{(b-b_1)\ugs}{2} \io \vepsxx^2
	+ \frac{(Bb_1-1)\ugs}{2b_1} \io \uepsxx^2
	+ \frac{b\del\ugs}{2} \io \uepsxx^2
	\qquad \mbox{for all } t\in (0,T).
  \eea
  Next, again by Young's inequality,
  \be{5.77}
	\bigg| \tau B \io \vepsx\wepsx \bigg|
	\le \frac{\al B_1}{2} \io \vepsx^2
	+ \frac{\tau^2 B^2}{2\al B_1} \io \wepsx^2
	\qquad \mbox{for all } t\in (0,T),
  \ee
  whence for all $t\in (0,T)$,
  \bea{5.8}
	\frac{\tau}{2} \io \wepsx^2
	+ \frac{\al B-\tau\del}{2} \io \vepsx^2
	+ \tau B \io \vepsx\wepsx
	&\ge& \frac{\tau}{2}\cdot\Big( 1-\frac{\tau B^2}{\al B_1}\Big) \io \wepsx^2
	+ \frac{\al(B-B_1) -\tau\del}{2} \io \vepsx^2 \nn\\
	&\ge& c_2 \io \wepsx^2
	+ \frac{\al(B-B_1)}{4} \io \vepsx^2
  \eea
  due to (\ref{5.50}) and (\ref{5.51}).
  Finally, two more times applying Young's inequality we see that
  \be{5.86}
	\bigg| \tau\del \io \uepsx\wepsx \bigg|
	\le \frac{c_2}{2} \io \wepsx^2
	+ \frac{\tau^2 \del^2}{2c_2} \io \uepsx^2
	\qquad \mbox{for all } t\in (0,T)
  \ee
  and
  \be{5.87}
	\bigg| \al \del \io \uepsx\vepsx \bigg|
	\le \frac{\al(B-B_1)}{4} \io \vepsx^2
	+ \frac{\al\del^2}{B-B_1} \io \uepsx^2
	\qquad \mbox{for all } t\in (0,T),
  \ee
  so that in line with (\ref{5.8}), the sum of the first-order expressions on the right of (\ref{5.4}) can all in all be controlled 
  according to
  \bea{5.9}
	& & \hs{-20mm}
	\frac{\tau}{2} \io \wepsx^2
	+ \frac{\al B-\tau\del}{2} \io \vepsx^2
	+ \tau B \io \vepsx\wepsx
	+ \tau\del \io \uepsx\wepsx
	+ \al\del \io \uepsx\vepsx \nn\\
	&\ge& \frac{c_2}{2} \io \wepsx^2
	- \Big(\frac{\tau^2}{2c_2} + \frac{\al}{B-B_1}\Big)\del^2 \io \uepsx^2
	\qquad \mbox{for all } t\in (0,T).
  \eea
  Since (\ref{5.33}) along with (\ref{5.52}) guarantees that
  \be{5.10}
	\Big(\frac{\tau^2}{2c_2} + \frac{\al}{B-B_1}\Big)\del^2 \io \uepsx^2
	\le
	\Big(\frac{\tau^2}{2c_2} + \frac{\al}{B-B_1}\Big)\del^2 c_1 \io \uepsxx^2
	\le \frac{b\del\ugs}{2} \io \uepsxx^2
	\qquad \mbox{for all } t\in (0,T),
  \ee
  and since (\ref{5.2}) clearly implies that
  \bas
	\eps \io \gaeps(\Teps) \uepsxxx^2
	\ge \ugs \eps \io \uepsxxx^2
	\qquad \mbox{for all } t\in (0,T),
  \eas
  by combining (\ref{5.7}) with (\ref{5.9}) we thus infer upon dropping a nonnegative summand in (\ref{y}) that
  \bas
	\yeps(t)
	\ge \frac{c_2}{2} \io \wepsx^2
	+ \frac{(b-b_1)\ugs}{2} \io \vepsxx^2
	+ \frac{(Bb_1-1)\ugs}{2b_1} \io \uepsxx^2
	+ \ugs \eps \io \uepsxxx^2
	\qquad \mbox{for all } t\in (0,T),
  \eas
  which in view of (\ref{5.6}) yields the first inequality in (\ref{5.3}).\abs
  To verify the second relation therein, we simply use the upper bound in (\ref{5.2}) in confirming by means of (\ref{5.66}),
  (\ref{5.77}), (\ref{5.86}) and (\ref{5.87}) that
  \bas
	\yeps(t)
	&\le& \frac{\tau}{2} \io \wepsx^2
	+ \frac{b\ogs}{2} \io \vepsxx^2
	+ \frac{(B+b\del)\ogs}{2} \io \uepsxx^2
	+ \frac{\al B}{2} \io \vepsx^2
	+ \frac{b_1\ogs}{2} \io \vepsxx^2
	+ \frac{\ogs}{2b_1} \io \uepsxx^2 \nn\\
	& & + \ogs \eps \io \uepsxxx^2
	+ \frac{(1+\tau)B\eps}{2} \io \vepsxx^2
	+ \frac{\al B_1}{2} \io \vepsx^2
	+ \frac{\tau^2 B^2}{2\al B_1} \io \wepsx^2 \nn\\
	& & + \frac{c_2}{2} \io \wepsx^2
	+ \frac{\tau^2 B^2}{2c_2} \io \uepsx^2
	+ \frac{\al(B-B_1)}{4} \io \vepsx^2
	+ \frac{\al \del^2}{B-B_1} \io \uepsx^2 \nn\\
	&=& \Big(\frac{\tau}{2} + \frac{\tau^2 \delta^2}{2\al B_1} + \frac{c_2}{2}\Big) \io \wepsx^2
	+ \Big( \frac{(b+b_1)\ogs}{2} + \frac{(1+\tau)B\eps}{2}\Big) \io \vepsxx^2 \nn\\
	& & + \Big( \frac{(B+b\delta)\ogs}{2} + \frac{\ogs}{2b_1}\Big) \io \uepsxx^2 
	+ \ogs\eps \io \uepsxxx^2 \nn\\
	& & + \frac{\al (3B+B_1)}{4} \io \vepsx^2
	+ \Big(\frac{\tau^2}{2c_2} + \frac{\al}{B-B_1}\Big) \del^2 \io \uepsx^2
	\qquad \mbox{for all } t\in (0,T).
  \eas
  Again using (\ref{5.10}), and employing (\ref{5.33}) to estimate
  \bas
	\frac{\al(3B+B_1)}{4} \io \vepsx^2
	\le \frac{c_1\al(3B+B_1)}{4} \io \vepsxx^2
	\qquad \mbox{for all } t\in (0,T),
  \eas
  from this we readily obtain the second inequality in (\ref{5.3}) with $k_2$ as in (\ref{5.65}), because
  $\frac{(1+\tau)B\eps}{2} \le \frac{(1+\tau)B}{2}$.
\qed
This property can be used to turn (\ref{4.1}) into a genuine Lyapunov-type differential inequality, still under an assumption
on the unknown solution itself but now referring to the behavior of $\Teps$ itself rather than an estimate of the form in (\ref{5.2}).
\begin{lem}\label{lem6}
  Assume that $\tau>0$, $\al>0$, $b>0$ and $B>0$ are such that (\ref{5.1}) holds,
  that $\gamma$ satisfies (\ref{gamma}), and that $\Ths>0$.
  Then there exist $\ugs=\ugs(\Ths,\gamma)>0$, $\ogs=\ogs(\Ths,\gamma)>\ugs$ and $\Gas=\Gas(\Ths,\gamma)>0$ 
  such that with $\del_1=\del_1(\tau,\al,b,B,\ugs(\Ths,\gamma),\ogs(\Ths,\gamma))$ as accordingly provided
  by Lemma \ref{lem5}, one can find $\del=\del(\tau,\al,b,B,\Ths,\gamma)\in (0,\del_1]$, 
  $\eta=\eta(\tau,\al,b,B,\Ths,\gamma) \in (0,1]$
  and $\kappa=\kappa(\tau,\al,b,B,\Ths,\gamma)>0$
  such that if (\ref{gaeps0}) and (\ref{gaepsc}) is satisfied, then
  for some $\epss=\epss\big(\Ths,(\gaeps)_{\eps\in (0,1)} \big) \in (0,1)$ it follows that
  whenever $D>0$ and (\ref{ie0}) holds and $\eps\in (0,\epss)$ as well as $T\in (0,\tme)$ have the property that
  \be{te}
	\Teps \le 2\Ths
	\quad \mbox{and} \quad
	|\Tepsx| + |\Tepsxx| + |\Tepst| \le \eta
	\qquad \mbox{in } \Om\times (0,T),
  \ee
  for $\yeps=\yeps^{(B,\del)}$ as in (\ref{y}) we have
  \be{6.3}
	\yeps'(t) + \kappa \yeps(t) + \frac{b\ugs}{2} \eps \io \vepsxxx^2 \le 0
	\qquad \mbox{for all } t\in (0,T),
  \ee
  and moreover (\ref{5.2}) holds as well as
  \be{6.33}
	|\gaeps'(\Teps)| \le \Gas
	\qquad \mbox{in } \Om\times (0,T).
  \ee
\end{lem}
\proof
  We let
  \bas
	\ugs\equiv \ugs(\Ths,\gamma):=\frac{1}{2} \inf_{\xi\in [0,2\Ths]} \gamma(\xi)
	\qquad \mbox{and} \qquad
	\ogs\equiv \ogs(\Ths,\gamma):=2\sup_{\xi\in [0,2\Ths]} \gamma(\xi),
  \eas
  and according to (\ref{gaepsc}) we can then fix $\epss=\epss\big(\Ths,(\gaeps)_{\eps\in (0,1)} \big) \in (0,1)$ in such a way that
  \be{6.44}
	\ugs \le \gaeps(\xi) \le \ogs
	\qquad \mbox{for all $\xi\in [0,2\Ths]$ and } \eps\in (0,\epss),
  \ee
  and that
  \be{6.5}
	|\gaeps'(\xi)| \le \Gas \equiv \Gas(\Ths,\gamma):= 1 + \sup_{\sig\in [0,2\Ths]} |\gaeps'(\sig)| 
	\quad \mbox{and} \quad
	|\gaeps''(\xi)| \le c_1\equiv c_1(\Ths,\gamma):=1 + \sup_{\sig\in [0,2\Ths]} |\gaeps''(\sig)| 
  \ee
  for all $\xi \in [0,2\Ths]$ and $\eps\in (0,\epss)$.
  Moreover noting that according to (\ref{5.1}), both
  \be{6.6}
	c_2\equiv c_2(\tau,\al,b,B):=\al-\tau B
	\qquad \mbox{and} \qquad
	c_3\equiv c_3(\tau,\al,b,B,\Ths,\gamma):=(bB-1)\ugs
  \ee
  are positive, we next let 
  $\del_1=\del_1(\tau,\al,b,B,\ugs(\Ths,\gamma),\ogs(\Ths,\gamma))$,
  $k_1=k_1(\tau,\al,b,B,\ugs(\Ths,\gamma),\ogs(\Ths,\gamma))$ and
  $k_2=k_2(\tau,\al,b,B,\ugs(\Ths,\gamma),\ogs(\Ths,\gamma))$ be as correspondingly obtained in Lemma \ref{lem5}, 
  and taking $c_4>0$ such that in line with
  a Poincar\'e inequality we have
  \be{6.7}
	\io \vp^2 \le c_4 \io \vp_x^2
	\qquad \mbox{for all } \vp\in W^{1,2}(\Om),
  \ee
  we define
  \be{6.8}
	\del\equiv \del(\tau,\al,b,B,\Ths,\gamma):=
	\min \bigg\{ \del_1 \, , \, \frac{c_3}{4c_4 \al} \, , \, \frac{c_3}{8(\frac{4\al^2}{\ugs}+\tau)} \, , \,
		\frac{\ugs}{(1+\tau)^2} \bigg\}
  \ee
  and
  \be{6.9}
	\kappa\equiv \kappa(\tau,\al,b,B,\Ths,\gamma):=
	\min \bigg\{ \frac{c_2}{2k_2} \, , \, \frac{c_3}{2k_2} \, , \, \frac{\del\ugs}{4k_2} \, , \, \frac{(B+b\del)\ugs}{2k_2}
		\bigg\}.
  \ee
  We finally abbreviate
  \be{6.10}
	c_5\equiv c_5(\tau,\al,b,B,\Ths,\gamma):=
	2\cdot\Big(\frac{b\Gas}{4} + \frac{\Gas}{4}\Big)  + \frac{bc_1}{4} + \frac{c_1}{4}
  \ee
  and
  \be{6.11}
	c_6\equiv c_6(\tau,\al,b,B,\Ths,\gamma):=
	\frac{b\Gas}{2} + \frac{\Gas}{4} + b\Gas + b\Gas c_4 + bc_1 c_4 + bB\Gas c_4 + \frac{bB \Gas}{4} + \frac{B\Gas}{4} + b\Gas c_4 \del
		+ \frac{\Gas}{4}
  \ee
  and
  \be{6.12}
	c_7\equiv c_7(\tau,\al,b,B,\Ths,\gamma):=
	2\Gas + \frac{(B+b\del)\Gas}{2} + \Gas c_4 + c_1 c_4 + b\Gas c_4 + \frac{b\Gas \del}{4} + \Gas c_4 \del + \frac{\Gas\del}{4}
		+ (B+b\del)\Gas
  \ee
  as well as
  \be{6.122}
	c_8\equiv c_8(\tau,\al,b,B,\Ths,\gamma):=
	2\Gas + \frac{B+b\del)\Gas}{4},
  \ee
  and fix some suitably small $\eta=\eta(\tau,\al,b,B,\Ths,\gamma) \in (0,1]$ satisfying
  \be{6.13}
	c_5 \eta \le \frac{c_2}{2},
	\quad
	\Big(c_6+\frac{b\Gas^2}{\ugs}\big) \eta \le \frac{c_3}{8},
	\quad
	\Big(c_7+\frac{\Gas^2}{b\ugs}\Big)\eta \le \frac{\del\ugs}{4}
	\quad \mbox{and} \quad
	\Big(c_8+\frac{\Gas^2}{2\ugs}\Big) \eta \le \frac{(B+b\del)\ugs}{2}.
  \ee
  Henceforth assuming that (\ref{te}) holds with some $\eps\in (0,\epss)$ and $T\in (0,\tme)$, we first draw on (\ref{6.44})
  and (\ref{6.5}) to see that indeed (\ref{5.2}) and (\ref{6.33}) hold, and that furthermore
  \be{6.14}
	|\gaeps''(\Teps)| \le c_1
	\qquad \mbox{in } \Om\times (0,T).
  \ee
  By positivity of the constants in (\ref{6.6}), from (\ref{4.1}), (\ref{5.3}) and (\ref{6.44}) we thus particularly obtain that
  \bea{6.15}
	& & \hs{-20mm}
	\yeps'(t)
	+ \kappa\yeps(t)
	+ c_2 \io \wepsx^2
	+ c_3 \io \vepsxx^2
	+ \del\ugs \io \uepsxx^2 \nn\\
	& & +\  \eps \io \wepsxx^2
	+ b\ugs \eps \io \vepsxxx^2
	+ (B+b\del) \ugs \eps \io \uepsxxx^2
	+ 2\ugs \eps^2 \io \uepsxxxx^2 \nn\\
	&\le& \kappa\cdot \bigg\{
	k_2 \io \wepsx^2
	+ k_2 \io \vepsxx^2
	+ k_2 \io \uepsxx^2
	+ k_2\eps \io \uepsxxx^2 \bigg\} \nn\\
	& & +\  \al\del \io \vepsx^2
	- 2\al\del\eps \io \uepsxx \vepsxx
	+ \tau\del\eps \io \vepsxx^2 \nn\\
	& & - (1+\tau)\del\eps \io \uepsxx\wepsxx + I_\eps(t) + J_\eps(t)
	\qquad \mbox{for all } t\in (0,T),
  \eea
  where
  \bea{6.16}
	I_\eps(t)
	&:=& \frac{b}{2} \io \gaeps'(\Teps) \Tepst \vepsxx^2
	+ \io \gaeps'(\Teps) \Tepst \uepsxx \vepsxx
	+ \frac{B+b\del}{2} \io \gaeps'(\Teps) \Tepst \uepsxx^2 \nn\\
	& & + b\io \gaeps'(\Teps) \Tepsx \vepsxx \wepsx
	+ \io \gaeps'(\Teps) \Tepsx \uepsxx\wepsx \nn\\
	& & + b\io \gaeps'(\Teps)\Tepsxx \vepsx \wepsx
	+ \io \gaeps'(\Teps) \Tepsxx \uepsx \wepsx \nn\\
	& & + b\io \gaeps''(\Teps) \Tepsx^2 \vepsx\wepsx
	+ \io \gaeps''(\Teps)\Tepsx^2 \uepsx\wepsx \nn\\
	& & - bB \io \gaeps'(\Teps) \Tepsx \vepsx \vepsxx
	- B \io \gaeps'(\Teps)\Tepsx \uepsx\vepsxx \nn\\
	& & - b\del \io \gaeps'(\Teps)\Tepsx \vepsx \uepsxx
	- \del \io \gaeps'(\Teps)\Tepsx\uepsx\uepsxx \nn\\
	& & + \eps \io \gaeps'(\Teps) \Tepst \uepsxxx^2
	- \eps \io \gaeps'(\Teps) \Tepsx \uepsxxx \vepsxx \nn\\
	& & -(B+b\del)\eps \io \gaeps'(\Teps) \Tepsx \uepsxx\uepsxxx,
	\qquad t\in (0,T),
  \eea
  and
  \bea{6.17}
	J_\eps(t)
	&:=&
	- b\eps \io \gaeps'(\Teps) \Tepsx \vepsxx\vepsxxx
	- \eps \io \gaeps'(\Teps)\Tepsx \uepsxx \vepsxxx \nn\\
	& & - 2\eps^2 \io \gaeps'(\Teps) \Tepsx \uepsxxx\uepsxxxx,
	\qquad t\in (0,T).
  \eea
  Here, (\ref{6.9}) asserts that
  \bas
	& & \hs{-20mm}
	 \kappa\cdot \bigg\{
	k_2 \io \wepsx^2
	+ k_2 \io \vepsxx^2
	+ k_2 \io \uepsxx^2
	+ k_2\eps \io \uepsxxx^2 \bigg\} \\
	&\le& \frac{c_2}{2} \io \wepsx^2
	+ \frac{c_3}{2} \io \vepsxx^2
	+ \frac{\del\ugs}{4} \io \uepsxx^2
	+ \frac{(B+b\del)\ugs\eps}{2} \io \uepsxxx^2
	\qquad \mbox{for all } t\in (0,T),
  \eas
  while (\ref{6.7}) along with Young's inequality and 
  the second, third and fourth smallness conditions contained in (\ref{6.8}) ensure that
  \bas
	\al\del\io \vepsx^2
	\le \al\del c_4 \io \vepsxx^2
	\le \frac{c_3}{4} \io \vepsxx^2
	\qquad \mbox{for all } t\in (0,T),
  \eas
  that
  \bas
	-2\al\del\eps \io \uepsxx\vepsxx
	+ \tau\del\eps \io \vepsxx^2
	&\le& \frac{\del\ugs}{4} \io \uepsxx^2
	+ \frac{4\al^2\del\eps^2}{\ugs} \io \vepsxx^2
	+ \tau\del\eps \io \vepsxx^2 \nn\\
	&\le& \frac{\del\ugs}{4} \io \uepsxx^2
	+ \Big(\frac{4\al^2}{\ugs} + \tau\big)\del \io \vepsxx^2 \nn\\
	&\le& \frac{\del\ugs}{4} \io \uepsxx^2
	+ \frac{c_3}{8} \io \vepsxx^2
	\qquad \mbox{for all } t\in (0,T),
  \eas
  and that
  \bas
	-(1+\tau)\del\eps \io \uepsxx\wepsxx
	&\le& \eps \io \wepsxx^2
	+ \frac{(1+\tau)^2 \del^2 \eps}{4} \io \uepsxx^2 \nn\\
	&\le& \eps \io \wepsxx^2
	+ \frac{(1+\tau)^2}{4} \del^2 \io \uepsxx^2 \nn\\
	&\le& \eps \io \wepsxx^2
	+ \frac{\del\ugs}{4} \io \uepsxx^2 
	\qquad \mbox{for all } t\in (0,T).
  \eas
  Therefore, (\ref{6.15}) implies that
  \bea{6.18}
	& & \hs{-20mm}
	\yeps'(t) + \kappa\yeps(t) + \frac{c_2}{2} \io \wepsx^2 + \frac{c_3}{8} \io \vepsxx^2 + \frac{\del\ugs}{4} \io \uepsxx^2 \nn\\
	& & +\ b\ugs \eps \io \vepsxxx^2
	+ \frac{(B+b\del)\ugs\eps}{2} \io \uepsxxx^2
	+ 2\ugs\eps^2 \io \uepsxxxx^2 \nn\\
	&\le& I_\eps(t) + J_\eps(t)
	\qquad \mbox{for all } t\in (0,T),
  \eea
  and to control the summands making up
  $I_\eps(t)$ here, we now repeatedly rely on (\ref{5.2}), (\ref{6.33}), (\ref{6.14}),
  (\ref{te}) and Young's inequality in verifying that
  \bas
	\frac{b}{2} \io \gaeps'(\Teps) \Tepst \vepsxx^2
	\le \frac{b\Gas}{2} \eta \io \vepsxx^2
  \eas
  and
  \bas
	\io \gaeps'(\Teps) \Tepst \uepsxx \vepsxx
	\le \Gas \eta \io |\uepsxx| \cdot |\vepsxx|
	\le \Gas \eta \io \uepsxx^2
	+ \frac{\Gas}{4} \eta \io \vepsxx^2
  \eas
  and
  \bas
	\frac{B+b\del}{2} \io \gaeps'(\Teps) \Tepst \uepsxx^2
	\le \frac{(B+b\del) \Gas}{2} \eta \io \uepsxx^2
  \eas
  and
  \bas
	b \io \gaeps'(\Teps) \Tepsx \vepsxx \wepsx
	\le b\Gas \eta \io |\vepsxx| \cdot |\wepsx|
	\le b\Gas \eta \io \vepsxx^2
	+ \frac{b\Gas}{4} \eta \io \wepsx^2
  \eas
  and
  \bas
	\io \gaeps'(\Teps) \Tepsx \uepsxx \wepsx
	\le \Gas \eta \io |\uepsxx| \cdot |\wepsx|
	\le \Gas\eta \io \uepsxx^2
	+ \frac{\Gas}{4} \eta \io \wepsx^2
  \eas
  for all $t\in (0,T)$, that thanks to (\ref{6.7}) we have
  \bas
	b \io \gaeps'(\Teps) \Tepsxx \vepsx \wepsx
	&\le& b\Gas\eta \io \vepsx^2
	+ \frac{b\Gas}{4} \eta \io \wepsx^2 \nn\\
	&\le& b\Gas c_4 \eta \io \vepsxx^2
	+ \frac{b\Gas}{4} \eta \io \wepsx^2 
  \eas
  and
  \bas
	\io \gaeps'(\Teps) \Tepsxx \uepsx \wepsx
	&\le& \Gas \eta \io \uepsx^2
	+ \frac{\Gas}{4} \eta \io \wepsx^2 \nn\\
	&\le& \Gas c_4 \eta \io \uepsxx^2
	+ \frac{\Gas}{4} \eta \io \wepsx^2
  \eas
  and, since $\eta\le 1$,
  \bas
	b \io \gaeps''(\Teps) \Tepsx^2 \vepsx \wepsx
	&\le& bc_1 \eta^2 \io \vepsx^2
	+ \frac{bc_1}{4} \eta^2 \io \wepsx^2 \nn\\
	&\le& bc_1 c_4 \eta \io \vepsxx^2
	+ \frac{bc_1}{4} \eta \io \wepsx^2 
  \eas
  as well as
  \bas
	\io \gaeps''(\Teps) \Tepsx^2 \uepsx\wepsx
	&\le& c_1\eta^2 \io \uepsx^2
	+ \frac{c_1}{4} \eta^2 \io \wepsx^2 \nn\\
	&\le& c_1 c_4 \eta \io \uepsxx^2
	+ \frac{c_1}{4} \eta \io \wepsx^2
  \eas
  for all $t\in (0,T)$.
  We similarly see that
  \bas
	- bB \io \gaeps'(\Teps) \Tepsx \vepsx \vepsxx
	\le bB \Gas c_4 \eta \io \vepsxx^2
	+ \frac{bB \Gas}{4} \eta \io \vepsxx^2
  \eas
  and
  \bas
	-B \io \gaeps'(\Teps) \Tepsx \uepsx \vepsxx
	\le B \Gas c_4 \eta \io \uepsxx^2
	+ \frac{B \Gas}{4} \eta \io \vepsxx^2
  \eas
  and
  \bas
	-b\del \io \gaeps'(\Teps) \Tepsx \vepsx \uepsxx
	\le b\Gas c_4 \del \eta \io \vepsxx^2
	+ \frac{b\Gas\del}{4} \eta \io \uepsxx^2
  \eas
  as well as
  \bas
	- \del \io \gaeps'(\Teps) \Tepsx \uepsx \uepsxx
	\le \Gas c_4 \del\eta \io \uepsxx^2
	+ \frac{\Gas\del}{4} \eta \io \uepsxx^2
  \eas
  for all $t\in (0,T)$, and that since $\eps\le 1$,
  \bas
	- \eps \io \gaeps'(\Teps) \Tepsx \uepsxxx\vepsxx
	\le \Gas \eta \eps \io \uepsxxx^2
	+ \frac{\Gas}{4} \eta \io \vepsxx^2
  \eas
  and
  \bas
	-(B+b\del) \eps \io \gaeps'(\Teps) \Tepsx \uepsxx \uepsxxx
	\le (B+b\del) \Gas \eta \io \uepsxx^2
	+ \frac{(B+b\del) \Gas}{4} \eta \eps \io \uepsxxx^2
  \eas
  as well as
  \bas
	 \eps \io \gaeps'(\Teps) \Tepst \uepsxxx^2
	\le \Gas \eta\eps \io \uepsxxx^2
  \eas
  for all $t\in (0,T)$.
  In view of (\ref{6.16}) and our definitions in (\ref{6.10})-(\ref{6.122}), the collection of these estimates reveals that
  \be{6.19}
	I_\eps(t)
	\le c_5 \eta \io \wepsx^2
	+ c_6 \eta \io \vepsxx^2
	+ c_7 \eta \io \uepsxx^2
	+ c_8 \eta \eps \io \uepsxxx^2
	\qquad \mbox{for all } t\in (0,T),
  \ee
  and in order to finally estimate the last summand in (\ref{6.18}), we make use of the artificial dissipation expressed in the last
  and third to last contributions to the left-hand side therein:
  Indeed, a combination of (\ref{6.33}), (\ref{6.14}) and (\ref{te}) 
  with Young's inequality shows that again since $\eta\le 1$ and $\eps\le 1$,
  \bas
	- b\eps \io \gaeps'(\Teps) \Tepsx \vepsxx \vepsxxx
	&\le& b\Gas \eta\eps \io |\vepsxx| \cdot |\vepsxxx| \nn\\
	&\le& \frac{b\ugs \eps}{4} \io \vepsxxx^2
	+ \frac{b\Gas^2 \eps}{\ugs} \eta^2 \io \vepsxx^2 \nn\\
	&\le& \frac{b\ugs \eps}{4} \io \vepsxxx^2
	+ \frac{b\Gas^2}{\ugs} \eta \io \vepsxx^2
  \eas
  and
  \bas
	- \eps \io \gaeps'(\Teps) \Tepsx \uepsxx \vepsxxx
	&\le& \Gas\eta\eps \io |\uepsxx| \cdot |\vepsxxx| \nn\\
	&\le& \frac{b\ugs\eps}{4} \io \vepsxxx^2
	+ \frac{\Gas^2\eps}{b\ugs} \eta^2 \io \uepsxx^2 \nn\\
	&\le& \frac{b\ugs\eps}{4} \io \vepsxxx^2
	+ \frac{\Gas^2}{b\ugs} \eta \io \uepsxx^2
  \eas
  as well as
  \bas
	-2\eps^2 \io \gaeps'(\Teps) \Tepsx \uepsxxx \uepsxxxx
	&\le& 2\Gas \eta\eps^2 \io |\uepsxxx| \cdot |\uepsxxxx| \nn\\
	&\le& 2\ugs\eps^2 \io \uepsxxxx^2
	+ \frac{\Gas^2 \eta^2\eps^2}{2\ugs} \io \uepsxxx^2 \nn\\
	&\le& 2\ugs\eps^2 \io \uepsxxxx^2
	+ \frac{\Gas^2}{2\ugs} \eta\eps \io \uepsxxx^2 
  \eas
  for all $t\in (0,T)$.
  Consequently,
  \bas
	J_\eps(t)
	\le \frac{b\ugs}{2} \eps \io \vepsxxx^2
	+ 2\ugs\eps^2 \io \uepsxxxx^2
	+ \frac{b\Gas^2}{\ugs} \eta \io \vepsxx^2
	+ \frac{\Gas^2}{b\ugs} \eta \io \uepsxx^2
	+ \frac{\Gas^2}{2\ugs} \eta\eps \io \uepsxxx^2
  \eas
  for all $t\in (0,T)$, so that (\ref{6.18}) together with (\ref{6.19}) shows that thanks to the smallness
  properties of $\eta$ listed in (\ref{6.13}), indeed the inequality in (\ref{6.3}) holds.
\qed
As an intermediate conclusion, from Lemma \ref{lem6} we immediately obtain some natural consequences upon simple integration
in time:
\begin{cor}\label{cor7}
  Let $\tau>0$, $\al>0$, $b>0$ and $B>0$ be such that (\ref{5.1}) holds,
  assume (\ref{gamma}),
  let $\Ths>0$, and let 
  $\eta=\eta(\tau,\al,b,B,\Ths,\gamma)$ and $\kappa=\kappa(\tau,\al,b,B,\Ths,\gamma)$ be as in Lemma \ref{lem6}.
  Then there exists $k_3=k_3(\tau,\al,b,B,\Ths,\gamma)>0$ with the property that if 
  $D>0$, if (\ref{gaeps0}), (\ref{gaepsc}) and (\ref{ie0}) hold, and if 
  $\eps\in (0,\epss)$ and $T\in (0,\tme)$ are such that (\ref{te}) is satisfied, it follows that
  \be{7.1}
	\io \wepsx^2(\cdot,t) \le k_3 \Aeps e^{-\kappa t}
	\qquad \mbox{for all } t\in (0,T)
  \ee
  and
  \be{7.2}
	\io \vepsxx^2(\cdot,t) \le k_3 \Aeps e^{-\kappa t}
	\qquad \mbox{for all } t\in (0,T)
  \ee
  and
  \be{7.3}
	\io \uepsxx^2(\cdot,t) \le k_3 \Aeps e^{-\kappa t}
	\qquad \mbox{for all } t\in (0,T),
  \ee
  and that
  \be{7.4}
	\eps \io \uepsxxx^2(\cdot,t) \le k_3 \Aeps e^{-\kappa t}
	\qquad \mbox{for all } t\in (0,T)
  \ee
  as well as
  \be{7.44}
	\eps \int_0^t e^{-\kappa(t-s)} \cdot \bigg\{ \io \vepsxxx^2(\cdot,s) \bigg\} ds \le k_3 \Aeps e^{-\kappa t}
	\qquad \mbox{for all } t\in (0,T),
  \ee
  where
  \be{7.5}
	\Aeps:=\io w_{0\eps x}^2
	+ \io v_{0\eps xx}^2 
	+ \io u_{0\eps xx}^2
	+ \eps \io u_{0\eps xxx}^2,
  \ee
  and where $\epss=\epss\big(\Ths,(\gaeps)_{\eps\in (0,1)} \big) \in (0,1)$ is taken from Lemma \ref{lem6}.
\end{cor}
\proof
  With $k_1=k_1(\tau,\al,b,B,\ugs(\Ths),\ogs(\Ths),\gamma)$ and $k_2=k_2(\tau,\al,b,B,\ugs(\Ths),\ogs(\Ths),\gamma)$ 
  as found in Lemma \ref{lem5}, an integration of (\ref{6.3}) using (\ref{5.3}) shows that
  if we let $\del=\del(\tau,\al,b,B,\Ths,\gamma)$ and $\kappa=\kappa(\tau,\al,b,B,\Ths,\gamma)$ be as in Lemma \ref{lem6}, then
  \bas
	& & \hs{-30mm}
	k_1 \io \wepsx^2
	+ k_1 \io \vepsxx^2
	+ k_1 \io \uepsxx^2
	+ k_1 \eps\io \uepsxxx^2 \\
	&\le& \yeps(t) \\
	&\le& \yeps(0) e^{-\kappa t} 
	- \frac{b\ugs}{2} \eps \int_0^t e^{-\kappa(t-s)} \cdot \bigg\{ \io \vepsxxx^2(\cdot,s) \bigg\} ds \\
	&\le& \bigg\{ k_2 \io w_{0\eps x}^2 + k_2 \io v_{0\eps xx}^2 + k_2 \io u_{0\eps xx}^2 + k_2 \eps \io u_{0\eps xxx}^2 \bigg\}
		\cdot e^{-\kappa t} \\
	& & 
	- \frac{b\ugs}{2} \eps \int_0^t e^{-\kappa(t-s)} \cdot \bigg\{ \io \vepsxxx^2(\cdot,s) \bigg\} ds
	\qquad \mbox{for all } t\in (0,T).
  \eas
  In line with (\ref{7.5}), this yields (\ref{7.1})-(\ref{7.4}) if we let 
  $k_3\equiv k_3(\tau,\al,b,B,\Ths,\gamma):=\max\{\frac{k_2}{k_1} \, , \, \frac{2k_2}{b\ugs}\}$.
\qed
\mysection{Controlling deviations from homogeneity in the temperature distribution}
Guided by the overall ambition to show that (\ref{te}) may in fact imply itself under appropriate assumptions,
we next attempt to quantify how far conclusions on smallness of $\weps,\veps$ and $\ueps$,
in the flavor of those derived under the hypothesis (\ref{te}) in the previous section, 
can in turn be used to make sure that $\Teps$ can indeed deviate from spatial homogeneity only to a moderate extent.\abs
As a first step toward this, let us examine how far the above implications entail smallness of certain source terms
arising in the heat subsystem of (\ref{0eps}) as well as a differentiated version thereof.
\begin{lem}\label{lem8}
  Assume that $\tau>0$, $\al>0$, $b>0$ and $B>0$ satisfy \eqref{5.1}, that \eqref{gamma} holds and
  $\Ths>0$, and that $\ugs=\ugs(\Ths,\gamma)$, $\ogs=\ogs(\Ths\gamma)$, 
  $\epss=\epss\big(\Ths,(\gaeps)_{\eps\in (0,1)} \big) \in (0,1)$,	
  $\eta=\eta(\tau,\al,b,B,\Ths\gamma)$ and $\kappa=\kappa(\tau,\al,b,B,\Ths\gamma)$ are as provided by Lemma \ref{lem6}.
  Then there exists $k_4=k_4(\tau,\al,b,B,\Ths\gamma)>0$ such that if $D>0$ and (\ref{gaeps0}), (\ref{gaepsc}) and (\ref{ie0}) hold,
  and if $\eps\in (0,\epss)$ and $T\in (0,\tme)$ are such that (\ref{te}) is valid, for
  \be{h}
	\heps(x,t):= b \gaeps(\Teps(x,t)) \vepsx^2(x,t),
	x\in\Om, \ t\in (0,T),
  \ee
  and $A_\eps$ as defined in Corollary \ref{cor7},
  we have
  \be{8.1}
	\|\hepsx(\cdot,t)\|_{L^2(\Om)} \le k_4 \Aeps e^{-\kappa t}
	\qquad \mbox{for all } t\in (0,T)
  \ee
  and
  \be{8.2}
	\|\heps(\cdot,t)\|_{L^\infty(\Om)} \le k_4 \Aeps e^{-\kappa t}
	\qquad \mbox{for all } t\in (0,T)
  \ee
  as well as
  \be{8.3}
	\int_0^t e^{\kappa s} \|\hepst(\cdot,s)\|_{L^2(\Om)}^2 ds \le k_4 \Aeps^2
	\qquad \mbox{for all } t\in (0,T).
  \ee

\end{lem}
\proof
  We fix $c_1>0$ such that
  \be{8.5}
	\|\vp\|_{L^\infty(\Om)} \le c_1 \|\vp_x\|_{L^2(\Om)} 
	\qquad \mbox{for all } \vp \in W_0^{1,2}(\Om),
  \ee
  and then employ (\ref{5.2}) along with (\ref{7.2}) to see that with $\ogs=\ogs(\Ths,\gamma)$ 
  and $k_3=k_3(\tau,\alpha,b,B,\Ths,\gamma)$ taken from Lemma \ref{lem6} and Corollary \ref{cor7}, respectively,
  \bea{8.6}
	\|\heps\|_{L^\infty(\Om)}
	&=& b\|\gaeps(\Teps)\|_{L^\infty(\Om)} \|\vepsx\|_{L^\infty(\Om)}^2 |\Om| \nn\\
	&\le& b\ogs \cdot c_1^2 \|\vepsxx\|_{L^2(\Om)}^2 |\Om| \nn\\
	&\le& b \ogs c_1^2 |\Om| \cdot k_3^2 \Aeps e^{-\kappa t}
	\qquad \mbox{for all } t\in (0,T).
  \eea
  Likewise, additionally drawing on (\ref{6.33}) and (\ref{te}) as well as the inequality $\eta\le 1$ we find that
  \bea{8.7}
	\|h_{\eps x}\|_{L^2(\Om)}
	&=& \big\| b\gaeps'(\Teps) \Tepsx \vepsx^2 + 2b\gaeps(\Teps) \vepsx \vepsxx \big\|_{L^2(\Om)} \nn\\
	&\le& b\Gas c_1^2 \|\vepsxx\|_{L^2(\Om)}^2 |\Om|^\frac{1}{2} 
	+ 2b\ogs \cdot c_1\|\vepsxx\|_{L^2(\Om)}^2 \nn\\
	&\le& \big\{ b\Gas c_1^2 |\Om|^\frac{1}{2} + 2b\ogs c_1 \big\} \cdot k_3 \Aeps e^{-\kappa t}
	\qquad \mbox{for all } t\in (0,T).
  \eea
  Finally, in view of the identity $\vepsxt=\eps\vepsxxx+\wepsx$ we may combine (\ref{5.2}), (\ref{6.33}), (\ref{te}),
  (\ref{8.5}) and (\ref{7.2}) with (\ref{7.1}) to estimate
  \bea{8.8}
	& & \hs{-20mm}
	\int_0^t e^{\kappa s} \|h_{\eps t}(\cdot,s)\|_{L^2(\Om)}^2 ds \nn\\
	&=& \int_0^t e^{\kappa s} \cdot \Big\| b\gaeps'(\Teps) \Tepst \vepsx^2
	+ 2b\gaeps(\Teps) \vepsx\cdot\big\{ \eps\vepsxxx+\wepsx\big\} \big\|_{L^2(\Om)}^2 ds \nn\\
	&\le& 3 \int_0^t e^{\kappa s} \cdot \big\| b\gaeps'(\Teps) \Tepst \vepsx^2 \big\|_{L^2(\Om)}^2 ds
	+ 3 \eps \int_0^t e^{\kappa s} \cdot \big\| 2b \gaeps(\Teps)\vepsx \vepsxxx \big\|_{L^2(\Om)}^2 ds \nn\\
	& & 3 \int_0^t e^{\kappa s} \cdot \big\| 2b\gaeps(\Teps) \vepsx \wepsx\big\|_{L^2(\Om)}^2 ds \nn\\
	&\le& 3b^2 \Gas^2 c_1^4 |\Om| \int_0^t e^{\kappa s} \|\vepsxx\|_{L^2(\Om)}^4 ds
	+ 12 b^2 \ogs^2 c_1^2 \eps \int_0^t e^{\kappa s} \|\vepsxx\|^2_{L^2(\Om)}\|\vepsxxx\|_{L^2(\Om)}^2 ds \nn\\
	& & + 12 b^2 \ogs^2 c_1^2 \int_0^t e^{\kappa s} \|\vepsxx\|_{L^2(\Om)}^2 \|\wepsx\|_{L^2(\Om)}^2 ds \nn\\
	&\le& 3b^2 \Gas^2 c_1^4 |\Om| \int_0^t e^{\kappa s} \cdot \big\{ k_3 \Aeps e^{-\kappa s}\big\}^2 ds 
	+ 12 b^2 \ogs^2 c_1^2 \eps \int_0^t e^{\kappa s} \cdot \big\{ k_3\Aeps e^{-\kappa s}\big\} \cdot 
		\bigg\{ \io \vepsxxx^2 \bigg\} ds \nn\\
	& & + 12 b^2 \ogs^2 c_1^2 \int_0^t e^{\kappa s} \cdot \big\{ k_3 \Aeps e^{-\kappa s}\big\} \cdot
		\big\{ k_3 \Aeps e^{-\kappa s}\big\} ds
	\qquad \mbox{for all } t\in (0,T),
  \eea
  where
  \bea{8.9}
	\int_0^t e^{\kappa s} \cdot \big\{ k_3 \Aeps e^{-\kappa s}\big\} \cdot
		\big\{ k_3 \Aeps e^{-\kappa s}\big\} ds
	&=& \int_0^t e^{\kappa s} \big\{ k_3 \Aeps e^{-\kappa s}\big\}^2 ds \nn\\
	&\le& k_3^2 \Aeps^2 \int_0^t e^{-\kappa s} ds \nn\\
	&\le& \frac{k_3^2\Aeps^2}{\kappa}
	\qquad \mbox{for all } t\in (0,T),
  \eea
  and where due to (\ref{7.44}) and the rough estimate $e^{-\kappa s} \le 1$ for $s\ge 0$,
  \bea{8.10}
	\eps \int_0^t e^{\kappa s} \cdot \big\{ k_3\Aeps e^{-\kappa s}\big\} \cdot 
		\bigg\{ \io \vepsxxx^2 \bigg\} ds 
	&\le& k_3 \Aeps \cdot \eps \int_0^t e^{\kappa s} \cdot \bigg\{ \io \vepsxxx^2 \bigg\}  ds \nn\\
	&\le& k_3^2 \Aeps^2
	\qquad \mbox{for all } t\in (0,T).
  \eea
  With some suitably large $k_4=k_4(\tau,\al,b,B,\Ths,\gamma)>0$, from (\ref{8.8})-(\ref{8.10}) we infer (\ref{8.3}),
  while (\ref{8.6}) and (\ref{8.7}) imply (\ref{8.2}) and (\ref{8.1}), respectively.
\qed
This preparation enables us to conclude the following by combining a parabolic comparison argument 
with a reasoning based on heat semigroup estimates.
\begin{lem}\label{lem9}
  Let $\tau>0$, $\al>0$, $b>0$ and $B>0$ satisfy (\ref{5.1}), 
  let $\Ths>0$, and let $\ugs=\ugs(\Ths,\gamma)$, $\ogs=\ogs(\Ths\gamma)$,
  $\epss=\epss\big(\Ths,(\gaeps)_{\eps\in (0,1)} \big) \in (0,1)$,	
  $\eta=\eta(\tau,\al,b,B,\Ths,\gamma)$ and $\kappa=\kappa(\tau,\al,b,B,\Ths,\gamma)$ be as in Lemma \ref{lem6}.
  Then for each $D>0$ there exist $k_5=k_5(D,\tau,\al,b,B,\Ths,\gamma)>0$ and 
  $\kappa_0=\kappa_0(D,\tau,\al,b,B,\Ths,\gamma)\in (0,\kappa]$
  such that if (\ref{gaeps0}), (\ref{gaepsc}) and (\ref{ie0}) are valid, and if
  $\eps\in (0,\epss)$ and $T\in (0,\tme)$ are such that (\ref{te}) are satisfied, then
  \be{9.1}
	\Teps(x,t)  \le \|\Theta_{0\eps}\|_{L^\infty(\Om)} + k_5 \Aeps
	\qquad \mbox{for all $x\in\Om$ and } t\in (0,T)
  \ee
  and
  \be{9.2}
	\|\Tepst(\cdot,t)\|_{L^\infty(\Om)}
	\le \|\Theta_{0\eps t}\|_{L^\infty(\Om)} + k_5 \Aeps
	\qquad \mbox{for all } t\in (0,T)
  \ee
  and
  \be{9.22}
	\|\Tepsxx(\cdot,t)\|_{L^\infty(\Om)}
	\le k_5\|\Theta_{0\eps t}\|_{L^\infty(\Om)} + k_5 \Aeps
	\qquad \mbox{for all } t\in (0,T)
  \ee
  as well as
  \be{9.3}
	\|\Tepsx(\cdot,t)\|_{L^\infty(\Om)}
	\le k_5 \cdot \big\{ \|\Theta_{0\eps x}\|_{L^\infty(\Om)} + \Aeps\big\} \cdot e^{-\kappa_0 t}
	\qquad \mbox{for all } t\in (0,T),
  \ee
  where we have set
  \be{9.4}
	\Theta_{0\eps t} := D \Theta_{0\eps xx} + b\gaeps(\Theta_{0\eps}) v_{0\eps x}^2,
  \ee
  and where $\Aeps$ is as in (\ref{7.5}).
\end{lem}
\proof
  We let $k_4=k_4(\tau,\al,b,B,\Ths,\gamma)$ be as in Lemma \ref{lem8}, and first introduce a comparison function by defining
  \bas
	\ov{\Theta}(x,t):=\|\Theta_{0\eps}\|_{L^\infty(\Om)} + \frac{k_4 \Aeps}{\kappa} \cdot (1-e^{-\kappa t}),
	\qquad x\in\bom, \ t\ge 0.
  \eas
  Then $\ov{\Theta}(x,0)=\|\Theta_{0\eps}\|_{L^\infty(\Om)} \ge \Theta_{0\eps}$ in $\Om$ and $\ov{\Theta}_x=0$
  on $\pO\times (0,\infty)$, so that since with $\heps$ taken from (\ref{h}) we have
  \bas
	\ov{\Theta}_t - D\ov{\Theta}_{xx} - \heps(x,t)
	= \ov{\Theta}_t - \heps(x,t) 
	= k_4 \Aeps e^{-\kappa t} - \heps(x,t)
	\ge 0
	\qquad \mbox{in } \Om\times (0,T)
  \eas
  by (\ref{8.2}), a comparison principle asserts that $\Teps\le\ov{\Theta}$ in $\Om\times (0,T)$ and thus, in particular,
  \be{9.44}
	\Teps(x,t) \le \|\Theta_{0\eps}\|_{L^\infty(\Om)} + \frac{k_4 \Aeps}{\kappa}
	\qquad \mbox{for all $x\in\Om$ and } t\in (0,T).
  \ee
  We next recall known smoothing properties of the Neumann and Dirichlet heat semigroups $(e^{t\DN})_{t\ge 0}$ and
  $(e^{t\DD})_{t\ge 0}$ (\cite{win_JDE2010}, \cite{quittner_souplet}) to fix $c_1=c_1(D)>0, c_2=c_2(D)>0$, $c_3=c_3(D)>0$ 
  and $\lam=\lam(D)>0$ such that for all $\vp\in C^0(\bom)$ we have
  \be{9.5}
	\|e^{tD\DN}\vp\|_{L^\infty(\Om)} \le c_1\cdot (1+t^{-\frac{1}{4}}) \|\vp\|_{L^2(\Om)}
	\qquad \mbox{for all } t>0
  \ee
  and
  \be{9.6}
	\|e^{tD\DD}\vp\|_{L^\infty(\Om)} \le c_2 e^{-\lam t} \|\vp\|_{L^\infty(\Om)}
	\qquad \mbox{for all } t>0
  \ee
  as well as
  \be{9.7}
	\|e^{tD\DN}\vp\|_{L^\infty(\Om)} \le c_3 t^{-\frac{1}{4}} e^{-\lam t} \|\vp\|_{L^2(\Om)}
	\qquad \mbox{for all } t>0.
  \ee
  On the basis of a Duhamel representation associated with the identity $\Theta_{\eps tt} = D\Theta_{\eps xxt} + \hepst$,
  using the maximum principle, (\ref{9.5}) and (\ref{8.3}) we then obtain that
  for all $t\in (0,T)$,
  \bas
	\|\Tepst(\cdot,t)\|_{L^\infty(\Om)}
	&=& \bigg\| e^{tD\DN} \Theta_{0\eps t} + \int_0^t e^{(t-s)D\DN} h_{\eps t}(\cdot,s) ds \bigg\|_{L^\infty(\Om)} \nn\\
	&\le& \|\Theta_{0\eps t}\|_{L^\infty(\Om)}
	+ c_1 \int_0^t \Big( 1+(t-s)^{-\frac{1}{4}} \Big) \|h_{\eps t}(\cdot,s)\|_{L^2(\Om)} ds \nn\\
	&\le& \|\Theta_{0\eps t}\|_{L^\infty(\Om)}
	+ c_1 \cdot \bigg\{ \int_0^t e^{\kappa s} \|h_{\eps t}(\cdot,s)\|_{L^2(\Om)}^2 ds \bigg\}^\frac{1}{2} \cdot
		\bigg\{ \int_0^t \Big(1+(t-s)^{-\frac{1}{4}}\Big)^2 ds \bigg\}^\frac{1}{2}.
  \eas
  Since
  \bas
	\int_0^t \Big(1+(t-s)^{-\frac{1}{4}}\Big)^2 e^{-\kappa s} ds
	&\le& 4 \int_0^{(t-1)_+} e^{-\kappa s} ds
	+ 4 \int_{t-1}^t (t-s)^{-\frac{1}{2}} ds \\
	&=& \frac{4}{\kappa} \cdot \big( 1-e^{-\kappa (t-1)_+} \big) + 8
	\qquad \mbox{for all } t>0,
  \eas
  in view of (\ref{8.3}) this implies that
  \be{9.8}
	\|\Tepst(\cdot,t)\|_{L^\infty(\Om)}
	\le \|\Theta_{0\eps}\|_{L^\infty(\Om)}
	+ c_1 k_4^\frac{1}{2} \cdot \Big(\frac{4}{\kappa}+8\Big) \cdot \Aeps
	\qquad \mbox{for all } t\in (0,T),
  \ee
  and thereby also entails that
  \bea{9.87}
	\|\Tepsxx(\cdot,t)\|_{L^\infty(\Om)}
	&=& \Big\| \frac{\Tepst(\cdot,t)}{D} - \frac{\heps(\cdot,t)}{D}\Big\|_{L^\infty(\Om)} \nn\\
	&\le& \frac{1}{D} \|\Theta_{0\eps t}\|_{L^\infty(\Om)}
	+ \frac{c_1 k_4^\frac{1}{2}}{D} \cdot \Big(\frac{4}{\kappa}+8\Big) \cdot \Aeps
	+ \frac{k_4 \Aeps}{D} \cdot e^{-\kappa t} \nn\\
	&\le& \frac{1}{D} \|\Theta_{0\eps t}\|_{L^\infty(\Om)}
	+ \Big\{ \frac{c_1 k_4^\frac{1}{2}}{D} \cdot \Big(\frac{4}{\kappa}+8\Big) + \frac{k_4}{D} \Big\} \cdot \Aeps
	\quad \mbox{for all } t\in (0,T).
  \eea
  Finally, using that $\Tepsx$ satisfies the homogeneous Dirichlet problem for $\Theta_{\eps xt}= D\Theta_{\eps xxx} + h_{\eps x}$
  in $\Om\times (0,\tme)$ with $\Tepsx|_{t=0}=\Theta_{0\eps x}$, in a similar manner we conclude from (\ref{9.6}), (\ref{9.7})
  and (\ref{8.1}) that
  \bea{9.9}
	\|\Tepsx(\cdot,t)\|_{L^\infty(\Om)}
	&=& \bigg\| e^{tD\DD} \Theta_{0\eps x} + \int_0^t e^{(t-s)D\DD} h_{\eps x}(\cdot,s) ds \bigg\|_{L^\infty(\Om)} \nn\\
	&\le& c_2 e^{-\lam t} \|\Theta_{0\eps x}\|_{L^\infty(\Om)}
	+ c_3 \int_0^t (t-s)^{-\frac{1}{4}} e^{-\lam (t-s)} \|h_{\eps x}(\cdot,s)\|_{L^2(\Om)} ds \nn\\
	&\le& c_2 e^{-\lam t} \|\Theta_{0\eps x}\|_{L^\infty(\Om)}
	+ c_3 k_4 \Aeps \int_0^t (t-s)^{-\frac{1}{4}} e^{-\lam (t-s)} e^{-\kappa s} ds
  \eea
  for all $t\in (0,T)$.
  Choosing any $\kappa_0=\kappa_0(D,\tau,\al,b,B,\Ths,\gamma)\in (0,\kappa]$ such that $\kappa_0<\lam$, we can here estimate
  \bas
	\int_0^t (t-s)^{-\frac{1}{4}} e^{-\lam (t-s)} e^{-\kappa s} ds
	&\le& \int_0^{(t-1)_+} e^{-\lam(t-s)} e^{-\kappa_0 s} ds
	+ e^{-\kappa_0(t-1)} \int_{t-1}^t (t-s)^{-\frac{1}{4}} ds \nn\\
	&=& \frac{e^{-\lam}}{\lam-\kappa_0} \cdot \big( e^{-\kappa_0(t-1)_+} - e^{-\lam (t-1)_+}\big)
	+ \frac{4}{3} e^{-\kappa_0 (t-1)} \\
	&\le& \Big( \frac{e^{\kappa_0-\lam}}{\lam-\kappa_0} + \frac{4 e^{\kappa_0}}{3}\Big) e^{-\kappa_0 t}
	\qquad \mbox{for all } t>0,
  \eas
  so that since $e^{-\lam t} \le e^{-\kappa_0 t}$ for all $t>0$, from (\ref{9.9}) we infer that
  \be{9.10}
	\|\Tepsx(\cdot,t)\|_{L^\infty(\Om)}
	\le \Big\{ c_2 \|\Theta_{0\eps x}\|_{L^\infty(\Om)}
	+ c_3 k_4 \cdot \Big( \frac{e^{\kappa_0-\lam}}{\lam-\kappa_0} + \frac{4 e^{\kappa_0}}{3}\Big) \cdot \Aeps \Big\}
		\cdot e^{-\kappa_0 t}
	\qquad \mbox{for all } t\in (0,T).
  \ee 
  With some adequately large $k_5=k_5(D,\tau,\al,b,B,\Ths,\gamma)>0$, 
  collecting (\ref{9.44}), (\ref{9.8}), (\ref{9.87}) and (\ref{9.10}) we obtain (\ref{9.1})-(\ref{9.3}).
\qed
\mysection{Closing the loop. Global relaxation in the approximate problems}
We are now in a position to perform a self-map type reasoning hence closing our loop of arguments:
\begin{lem}\label{lem10}
  Let 
  $\tau>0$, $\al>0$ and $b>0$ be such that (\ref{12.01}) holds, let $D>0$ and $\Ths>0$, and assume $\gamma$ to comply with 
  (\ref{gamma}).
  Then there exist $\nu=\nu(D,\tau,\al,b,\Ths,\gamma)>0$,
  $\ks=\ks(D,\tau,\al,b,\Ths,\gamma)>0$ and $k_6=k_6(D,\tau,\al,b,\Ths,\gamma)>1$
  such that whenever (\ref{gaeps0}), (\ref{gaepsc}), (\ref{init}), (\ref{ie0}), (\ref{ie1}), (\ref{iec}) and (\ref{iec2}) hold with
  \be{10.1}
	\io u_{0xx}^2 + \io (u_{0t})_{xx}^2 + \io (u_{0tt})_x^2 < \nu
  \ee
  as well as
  \be{10.2}
	\|\Theta_0\|_{L^\infty(\Om)} < \Ths
	\qquad \mbox{and} \qquad
	\|\Theta_{0x}\|_{L^\infty(\Om)} + \|\Theta_{0xx}\|_{L^\infty(\Om)} < \nu,
  \ee
  one can fix 
  $\eps_0=\eps_0\big(D,\tau,\al,b,(u_{0\eps})_{\eps\in (0,1)},(v_{0\eps})_{\eps\in (0,1)},(w_{0\eps})_{\eps\in (0,1)},
  (\Theta_{0\eps})_{\eps\in (0,1)},(\gaeps)_{\eps\in (0,1)}\big) \in (0,1)$
  in such a way that whenever $\eps\in (0,\eps_0)$, in Lemma \ref{lem_loc} we have $\tme=\infty$ and
  \be{10.3}
	\io \wepsx^2(\cdot,t) + \io \vepsxx^2(\cdot,t) + \io \uepsxx^2(\cdot,t)
	\le C e^{-\ks t}
	\qquad \mbox{for all } t>0
  \ee
  and
  \be{10.33}
	\frac{1}{k_6} \le \gaeps(\Teps) \le k_6
	\qquad \mbox{in } \Om\times (0,\infty)
  \ee
  and
  \be{10.4}
	\|\Teps(\cdot,t)\|_{W^{2,\infty}(\Om)} \le k_6
	\qquad \mbox{for all } t>0
  \ee
  and
  \be{10.44}
	\|\Tepst(\cdot,t)\|_{L^\infty(\Om)} \le k_6
	\qquad \mbox{for all } t>0
  \ee
  as well as
  \be{10.5}
	\|\Tepsx(\cdot,t)\|_{L^\infty(\Om)} \le k_6 e^{-\ks t}
	\qquad \mbox{for all } t>0
  \ee
  and
  \be{10.55}
	0 \le \frac{d}{dt} \io \Teps(\cdot,t) \le k_6 e^{-\ks t}
	\qquad \mbox{for all } t>0.
  \ee
\end{lem}
\proof
  We choose a constant $c_1>0$ such that
  \be{10.6}
	\|\vp\|_{L^\infty(\Om)} \le c_1 \|\vp_x\|_{L^2(\Om)}
	\qquad \mbox{for all $\vp\in W^{1,2}(\Om)$ satisfying } \inf_\Om \vp=0,
  \ee
  and taking any $B=B(\tau,\al,b)>0$ fulfilling (\ref{5.1}), we let 
  $\del=\del(\tau,\al,b,B,\Ths,\gamma)$,
  $\eta=\eta(\tau,\al,b,B,\Ths,\gamma)$,
  $\kappa=\kappa(\tau,\al,b,B,\Ths,\gamma)$,
  $\kappa_0=\kappa_0(\tau,\al,b,B,\Ths,\gamma)$,
  $k_3=k_3(\tau,\al,b,B,\Ths,\gamma)$,
  $k_4=k_4(\tau,\al,b,B,\Ths,\gamma)$ and
  $k_5=k_5(D,\tau,\al,b,B,\Ths,\gamma)$
  be as provided by Lemma \ref{lem6}, Corollary \ref{cor7}, Lemma \ref{lem8} and Lemma \ref{lem9}, respectively.
  We then fix $\nu=\nu(D,\tau,\al,b,\Ths,\gamma) \in(0,1]$ suitably small such that
  \be{10.7}
	k_5 \nu \le \frac{1}{2} \Ths,
  \ee
  and that writing
  \be{10.77}
	c_2\equiv c_2(\Ths,\gamma) := \max \bigg\{ \frac{2}{\inf_{\xi\in [0,\Ths]} \gamma(\xi)} \, , \, 
		1 + \sup_{\xi\in [0,\Ths]} \gamma(\xi) \bigg\}
  \ee
  as well as
  \be{10.78}
	c_3\equiv c_3(D,\Ths,\gamma):=D+bc_1^2 c_2
  \ee
  we have
  \be{10.79}
	(1+c_3)\nu \le \frac{\eta}{2}
  \ee
  and
  \be{10.8}
	\big\{ 4k_5 + (k_5+1) c_3 \big\} \cdot \nu
	\le \frac{\eta}{2}.
  \ee
  Henceforth assuming (\ref{10.1}), (\ref{10.2}), (\ref{init}), (\ref{ie0}), (\ref{ie1}), (\ref{iec}), (\ref{iec2}), (\ref{gaeps0})  
  and (\ref{gaepsc}),
  we then infer that letting $\epss=\epss\big(\Ths,(\gaeps)_{\eps\in (0,1)}\big)$ be as obtained in Lemma \ref{lem6}, 
  we can find 		
  $\eps_0=\eps_0\big(D,\tau,\al,b,(u_{0\eps})_{\eps\in (0,1)},$ $(v_{0\eps})_{\eps\in (0,1)},(w_{0\eps})_{\eps\in (0,1)},
  (\Theta_{0\eps})_{\eps\in (0,1)},(\gaeps)_{\eps\in (0,1)}\big) \in (0,\epss]$ such that whenever $\eps\in (0,\eps_0)$,
  \be{10.9}
	\io w_{0\eps x}^2 + \io v_{0\eps xx}^2 + \io u_{0\eps xx}^2 + \eps \io u_{0\eps xxx}^2 \le \nu
  \ee
  and
  \be{10.10}
	\|\Theta_{0\eps}\|_{L^\infty(\Om)} \le \Ths
	\qquad \mbox{and} \qquad
	\|\Theta_{0\eps x}\|_{L^\infty(\Om)} + \|\Theta_{0\eps xx}\|_{L^\infty(\Om)} \le \nu
  \ee
  as well as
  \be{10.100}
	\frac{1}{c_2} \le \gaeps(\xi) \le c_2
	\qquad \mbox{for all } \xi\in [0,\Ths].
  \ee
  Due to (\ref{10.6}), (\ref{10.77}) and (\ref{10.78}), this especially ensures that the functions in (\ref{9.4}) satisfy
  \bea{10.11}
	\|\Theta_{0\eps t}\|_{L^\infty(\Om)}
	&=& \big\| D\Theta_{0\eps xx} + b\gaeps(\Theta_{0\eps}) v_{0\eps x}^2 \big\|_{L^\infty(\Om)} \nn\\
	&\le& D\|\Theta_{0\eps xx}\|_{L^\infty(\Om)}^2 + bc_2 c_1^2 \|v_{0\eps xx}\|_{L^2(\Om)}^2 \nn\\
	&\le& D\nu^2 + bc_2 c_1^2 \nu \nn\\
	&\le& c_3 \nu
	\qquad \mbox{for all } \eps\in (0,\eps_0),
  \eea
  because $v_{0\eps x}$ hs compact support in $\Om$ and $\nu\le 1$.\abs
  For $\eps\in (0,\eps_0)$, we now let
  \bea{10.12}
	T_\eps
	&:=& \sup \Big\{ T'\in (0,\tme) \ \Big| \ 
	\|\Teps(\cdot,t)\|_{L^\infty(\Om)} < 2\Ths \mbox{ and } \nn\\
	& & \hs{0mm}
	\|\Tepsx(\cdot,t)\|_{L^\infty(\Om)} + \|\Tepsxx(\cdot,t)\|_{L^\infty(\Om)} + \|\Tepst(\cdot,t)\|_{L^\infty(\Om)} < \eta 
	\mbox{ for all } t\in (0,T') \Big\},
  \eea
  and note that thanks to the continuity of $\Teps,\Tepsx,\Tepsxx$ and $\Tepst$ in $\bom\times [0,\tme)$,
  $T_\eps$ is well-defined with $T_\eps\in (0,\infty]$ due to the first inequality in (\ref{10.10}) and the fact that
  \be{10.13}
	\|\Theta_{0\eps x}\|_{L^\infty(\Om)} 
	+ \|\Theta_{0\eps xx}\|_{L^\infty(\Om)} 
	+ \|\Theta_{0\eps t}\|_{L^\infty(\Om)} 
	\le \nu + c_3 \nu \le \frac{\eta}{2}
  \ee
  by the second restriction in (\ref{10.10}), by (\ref{10.11}), and by (\ref{10.79}).\abs
  Now according to (\ref{10.12}) and our selection of $\eta$, Lemma \ref{lem9} applies so as to assert that since the number 
  $\Aeps$ in (\ref{7.5}) satisfies $\Aeps\le\nu$ by (\ref{10.9}), due to (\ref{10.10}) and (\ref{10.7}) for $\eps\in(0,\eps_0)$ 
  we have
  \be{10.14}
	\|\Teps(\cdot,t)\|_{L^\infty(\Om)}
	\le \|\Theta_{0\eps}\|_{L^\infty(\Om)} + k_5 \nu
	\le \Ths + k_5 \nu
	\le \frac{3\Ths}{2}
	\qquad \mbox{for all } t\in (0,T_\eps),
  \ee
  while by (\ref{10.10}) and the choice of $\kappa_0$,
  \bea{10.15}
	\|\Tepsx(\cdot,t)\|_{L^\infty(\Om)}
	&\le& k_5 \cdot \big\{ \|\Theta_{0\eps x} \|_{L^\infty(\Om)} + \nu \big\} \cdot e^{-\kappa_0 t} \nn\\
	&\le& 2k_5 \nu e^{-\kappa_0 t}
	\qquad \mbox{for all $t\in (0,T_\eps)$ and } \eps\in(0,\eps_0)
  \eea
  and by (\ref{10.11}) and (\ref{10.10}),
  \bea{10.16}
	\|\Tepsxx(\cdot,t)\|_{L^\infty(\Om)}
	+ \|\Tepst(\cdot,t)\|_{L^\infty(\Om)}
	&\le& \big\{ k_5 \|\Theta_{0\eps t}\|_{L^\infty(\Om)} + k_5 \nu \big\} 
	+ \big\{ \|\Theta_{0\eps t}\|_{L^\infty(\Om)} + k_5 \nu \big\} \nn\\
	&=& (k_5+1) \|\Theta_{0\eps t}\|_{L^\infty(\Om)} + 2 k_5 \nu \nn\\
	&\le& \big\{ (k_5+1) c_3 + 2 k_5 \big\}\cdot  \nu
  \eea
  for all $t\in (0,T_\eps)$ and $\eps\in(0,\eps_0)$.
  In particular, a combination of (\ref{10.15}) with (\ref{10.13}) shows that in line with (\ref{10.8}),
  \bas
	\|\Tepsx(\cdot,t)\|_{L^\infty(\Om)}
	+ \|\Tepsxx(\cdot,t)\|_{L^\infty(\Om)}
	+ \|\Tepst(\cdot,t)\|_{L^\infty(\Om)}
	&\le& \big\{ 4k_5 + (k_5+1) c_3 \big\} \cdot \nu \nn\\
	&\le& \frac{\eta}{2}
	\qquad \mbox{for all $t\in (0,T_\eps)$ and }\eps\in(0,\eps_0),
  \eas
  which in conjunction with (\ref{10.14}) implies that, again by continuity of $\Teps,\Tepsx,\Tepsxx$ and $\Tepst$, we actually must 
  have
  \be{10.17}
	T_\eps=\tme.
  \ee
  To conclude from this as intended, we finally invoke Corollary \ref{cor7} and thereby see that since
  $\kappa_0\le\kappa$, and again since
  $\Aeps \le\nu$, from (\ref{10.12}) and (\ref{10.17}) it follows that
  \be{10.18}
	\io \wepsx^2 + \io \vepsxx^2 + \io \uepsxx^2 \le 3k_3 \nu e^{-\kappa_0 t}
	\qquad \mbox{for all } t\in (0,\tme),
  \ee
  which together with the bound for $\Teps$ entailed by (\ref{10.12}) ensures that, in fact, in (\ref{ext_eps}) the second
  alternative cannot occur. Therefore, indeed, $\tme=\infty$, so that since Lemma \ref{lem8} accordingly says that
  with $\heps$ as in (\ref{h}) we have
  \be{10.19}
	\frac{d}{dt} \io \Teps
	= \io \heps
	\le k_4 \Aeps |\Om| e^{-\kappa_0 t}
	\le k_4 \nu |\Om| e^{-\kappa_0 t}
	\qquad \mbox{for all } t>0,
  \ee
  we obtain (\ref{10.3}), (\ref{10.33}), (\ref{10.4}), (\ref{10.44}), (\ref{10.5}) and (\ref{10.55}) as consequences of (\ref{10.18}),
  the definition of $T_\eps$, (\ref{10.100}), (\ref{10.15}) and (\ref{10.19}) if we let 
  $\ks\equiv \ks(D,\tau,\al,b,\Ths,\gamma):=\kappa_0$,
  and choose $k_6=k_6(D,\tau,\al,b,\Ths,\gamma)>0$ appropriately large.
\qed
\mysection{Passing to the limit $\eps\searrow 0$. Conclusion}\label{sec7}
It remains to appropriately pass to the limit $\eps\searrow 0$ in order to achieve our main result on global existence
of strong solutions decaying exponentially in the large time limit.\abs
\proofc of Theorem \ref{theo12}. \quad
  The statement on uniqueness has been covered by \cite{claes_win2}, so that it remains to construct a strong solution
  with the claimed properties.
  To this end, relying on (\ref{12.01}) we first employ Lemma \ref{lem10} to see that if we let 
  $\nu=\nu(D,\tau,\al,b,\Ths,\gamma)$,
  $\ks=\ks(D,\tau,\al,b,\Ths,\gamma)$ and
  $k_6=k_6(D,\tau,\al,b,\Ths,\gamma)$ be as provided there, and if we let 
  $u_0, u_{0t}, u_{0tt}$ and $\Theta_0$ satisfy (\ref{init}), (\ref{12.1}) and (\ref{12.2}), then whenever
  (\ref{gaeps0}), (\ref{gaepsc}), (\ref{ie0}), (\ref{ie1}), (\ref{iec}) and (\ref{iec2}) hold, we can find a constant
  $\eps_0=\eps_0\big(D,\tau,\al,b,(u_{0\eps})_{\eps\in (0,1)},(v_{0\eps})_{\eps\in (0,1)},(w_{0\eps})_{\eps\in (0,1)},
  (\Theta_{0\eps})_{\eps\in (0,1)},(\gaeps)_{\eps\in (0,1)}\big) \in (0,1)$
  such that for each $\eps\in (0,\eps_0)$,
  the solution of (\ref{0eps}) is global in time and satisfies (\ref{10.3})-(\ref{10.55}).\abs
  Since for any $\psi\in C^1(\bom)$ we can use the first equation in (\ref{0eps}) and the Cauchy-Schwarz inequality to estimate
  \bas
	\bigg| \tau \io \wepst \psi \bigg|
	&=& \bigg| - \eps \io \wepsx \psi_x 
	- b \io \gaeps(\Teps)\vepsx\psi_x
	- \io \gaeps(\Teps)\uepsx \psi_x -\alpha\io \weps\psi\bigg| \nn\\
	&\le& \eps \|\wepsx\|_{L^2(\Om)} \|\psi_x\|_{L^2(\Om)}
	+ b \|\gaeps(\Teps)\|_{L^\infty(\Om)} \|\vepsx\|_{L^2(\Om)} \|\psi_x\|_{L^2(\Om)} +\alpha\|\weps\|_{L^2(\Om)}\|\psi\|_{L^2(\Om)}\\
	& & + \|\gaeps(\Teps)\|_{L^\infty(\Om)} \|\uepsx\|_{L^2(\Om)} \|\psi_x\|_{L^2(\Om)}	
	\qquad \mbox{for all $t>0$ and } \eps\in (0,1),
  \eas
  from (\ref{10.3}), (\ref{10.33}), 
  (\ref{mass})
  and a Poincar\'e type inequality we particularly obtain that
  \bas
	(\weps)_{\eps\in (0,\eps_0)} 
	\mbox{ is bounded in } L^\infty((0,\infty);W^{1,2}(\Om)),
  \eas
  and that
  \bas
	(\wepst)_{\eps\in (0,\eps_0)} 
	\mbox{ is bounded in } L^\infty\big((0,\infty);(W^{1,2}(\Om))^\star\big),
  \eas
  while from (\ref{10.3}) as well as the second and third equations in (\ref{0eps}) it directly follows that
  \bas
	(\veps)_{\eps\in (0,\eps_0)} \ \mbox{and} \ (\ueps)_{\eps\in (0,\eps_0)}
	\mbox{ are bounded in } L^\infty((0,\infty);W^{2,2}(\Om)),
  \eas
  and that
  \bas
	(\vepst)_{\eps\in (0,\eps_0)} \ \mbox{and} \ (\uepst)_{\eps\in (0,\eps_0)}
	\mbox{ are bounded in } L^\infty((0,\infty);L^2(\Om)).
  \eas
  As from (\ref{10.4}) and (\ref{10.44}) we know that furthermore
  \bas
	(\Teps)_{\eps\in (0,\eps_0)} 
	\mbox{ is bounded in } L^\infty((0,\infty);W^{2,\infty}(\Om)),
  \eas
  and that
  \bas
	(\Tepst)_{\eps\in (0,\eps_0)} 
	\mbox{ is bounded in } L^\infty(\Om\times (0,\infty)),
  \eas
  relying on the compactness of the embeddings $W^{1,2}(\Om) \hra C^0(\bom)$ and $W^{2,2}(\Om)\hra C^1(\bom)$ we may therefore
  apply an Aubin-Lions lemma (\cite{simon}) to obtain $(\eps_j)_{j\in\N}\subset (0,\eps_0)$ as well as functions
  \be{12.5}
	\lbal
	w \in C^0(\bom\times [0,\infty)) \cap L^\infty((0,\infty);W^{1,2}(\Om)), \\[1mm]
	v \in C^0([0,\infty);C^1(\bom)) \cap L^\infty((0,\infty);W^{2,2}_N(\Om)), \\[1mm]
	u \in C^0([0,\infty);C^1(\bom)) \cap L^\infty((0,\infty);W^{2,2}_N(\Om))
	\qquad \mbox{and} \\[1mm]
	\Theta \in C^0([0,\infty);C^1(\bom)) \cap L^\infty((0,\infty);W^{2,\infty}_N(\Om))
	\quad \mbox{with} \quad
	\Theta_t \in L^\infty(\Om\times (0,\infty))
	\ear
  \ee
  such that $\eps_j\searrow 0$ as $j\to\infty$, that $\Theta\ge 0$ in $\Om\times (0,\infty)$, and that
  \begin{eqnarray}
	& & \weps \to w
	\qquad \mbox{in } C^0_{loc}(\bom\times [0,\infty)),
	\label{12.6} \\
	& & \wepsx \wto w_x
	\qquad \mbox{in } L^2_{loc}(\bom\times [0,\infty)),
	\label{12.7} \\
	& & \veps \to v
	\qquad \mbox{in } C^0_{loc}([0,\infty);C^1(\bom)),
	\label{12.8} \\
	& & \ueps \to u
	\qquad \mbox{in } C^0_{loc}([0,\infty);C^1(\bom))
	\qquad \qquad \mbox{and}
	\label{12.9} \\
	& & \Teps \to \Theta
	\qquad \mbox{in } C^0_{loc}([0,\infty);C^1(\bom))
	\label{12.10}
  \end{eqnarray}
  as $\eps=\eps_j\searrow 0$.
  Since from the third equation in (\ref{0eps}) it follows that for each $\vp\in C_0^\infty(\Om\times [0,\infty))$ 
  and any $\eps\in (0,1)$ we have
  \bas
	- \int_0^\infty \io \ueps \vp_t 
	- \io u_{0\eps} \vp(\cdot,0)
	= - \eps \int_0^\infty \io \uepsx \vp_x
	+ \int_0^\infty \io  \veps \vp,
  \eas
  using (\ref{12.9}), (\ref{12.8}) and (\ref{iec}) we infer on letting $\eps=\eps_j\searrow 0$ here that
  \bas
	- \int_0^\infty \io u\vp_t - \io u_0 \vp(\cdot,0) = \int_0^\infty \io v \vp
  \eas
  for any such $\vp$, meaning that $u_t=v$ in $\Om\times (0,\infty)$ and $u(\cdot,0)=u_0$ in $\Om$.
  Similarly combining the second equation in (\ref{0eps}) with (\ref{12.8}) and (\ref{12.6}), we obtain that $v_t=w$ in 
  $\Om\times (0,\infty)$ and $v(\cdot,0)=u_{0t}$ in $\Om$, so that (\ref{12.5}) implies that
  \be{12.11}
	\lbal
	u_t \in C^0([0,\infty);C^1(\bom)) \cap L^\infty((0,\infty);W^{2,2}_N(\Om))	
	\qquad \mbox{and} \\[1mm]
	u_{tt} \in C^0(\bom\times [0,\infty)) \cap L^\infty((0,\infty);W^{1,2}(\Om))	
	\ear
  \ee
  with $(u,u_t)(\cdot,0)=(u_0,u_{0t})$ in $\Om$.\abs
  Apart from that, given any $\vp\in C_0^\infty(\bom\times [0,\infty))$ we see using the first equation in (\ref{0eps}) that
  \bas
	- \tau \int_0^\infty \io \weps \vp_t - \tau \io w_{0\eps} \vp(\cdot,0)
	= - \eps \int_0^\infty \io \wepsx \vp_x
	&-& b \int_0^\infty \io \gaeps(\Teps) \vepsx \vp_x \\
	&-& \int_0^\infty \io \gaeps(\Teps) \uepsx \vp_x - \int_0^\infty \io \weps \vp
  \eas
  for all $\eps\in (0,1)$, and may utilize (\ref{12.6}), (\ref{iec}), (\ref{12.7}), (\ref{12.10}), (\ref{gaepsc}), (\ref{12.8})
  and (\ref{12.9}) to find that taking $\eps=\eps_j\searrow 0$ here leads to the identity
  \bas
	- \tau \int_0^\infty \io w \vp_t - \tau \io u_{0tt} \vp(\cdot,0)
	= - b \int_0^\infty \io \gamma(\Theta) v_x \vp_x
	- \int_0^\infty \io \gamma(\Theta) u_x \vp_x - \int_0^\infty \io w \vp,
  \eas
  which in light of the fact that $w=u_{tt}$ and $v_x=u_{xt}$ is equivalent to (\ref{wu}).\abs
  Likewise, from (\ref{12.10}), (\ref{iec}), (\ref{gaepsc}) and (\ref{12.8}) we obtain that if 
  $\vp\in C_0^\infty(\bom\times [0,\infty))$, then in the equation
  \bas
	- \int_0^\infty \io \Teps \vp_t - \io \Theta_{0\eps} \vp(\cdot,0)
	= - D \int_0^\infty \io \Tepsx \vp_x
	+ b \int_0^\infty \io \gaeps(\Teps) \vepsx^2 \vp,
  \eas
  valid for all $\eps\in (0,1)$ due to (\ref{0eps}), we may let $\eps=\eps_j\searrow 0$ to see that
  \bas
	- \int_0^\infty \io \Theta \vp_t - \io \Theta_0 \vp(\cdot,0)
	= - D \int_0^\infty \io \Theta_x \vp_x
	+ b \int_0^\infty \io \gamma(\Theta) u_{xt}^2 \vp.
  \eas
  Therefore, $\Theta$ forms a weak solution, in the sense specified in (\cite{LSU}), of
  \be{12.12}
	\lball
	\Theta_t = D\Theta_{xx} + h(x,t),
	\qquad & x\in\Om, \ t>0, \\[1mm]
	\frac{\pa\Theta}{\pa\nu}=0,
	\qquad & x\in\pO, \ t>0, \\[1mm]
	\Theta(x,0)=\Theta_0(x),
	\qquad & x\in\Om,
	\ear
  \ee
  with $h(x,t):=b\gamma(\Theta) u_{xt}^2$.
  Now 	
  from (\ref{12.11}) we particularly obtain that 
  $u_{xtt} \in L^\infty((0,\infty);L^2(\Om))$, whence by means of a Gagliardo-Nirenberg interpolation and the mean value theorem
  we infer that with some $c_1>0$, $c_2>0$ and $c_3>0$,
  \bas
	\|u_{xt}(\cdot,t)-u_{xt}(\cdot,s)\|_{L^\infty(\Om)}
	&\le& c_1 
	\|u_{xt}(\cdot,t)-u_{xt}(\cdot,s)\|_{W^{1,2}(\Om)}^\frac{1}{2}
	\|u_{xt}(\cdot,t)-u_{xt}(\cdot,s)\|_{L^2(\Om)}^\frac{1}{2} \\
	&\le& c_2
	\bigg\| \int_s^t u_{xtt}(\cdot,\sig) d\sig \bigg\|_{L^2(\Om)}^\frac{1}{2} \\
	&\le& c_3 |t-s|^\frac{1}{2}
	\qquad \mbox{for all $t>0$ and $s>0$.}
  \eas
  As moreover $u_{xt} \in C^0(\bom\times [0,\infty)) \cap L^\infty((0,\infty);C^\frac{1}{2}(\bom))$ according to (\ref{12.11})
  and the continuity of the embedding $W^{2,2}(\Om) \hra C^\frac{3}{2}(\bom)$, if thus follows that
  $u_{xt} \in C^{\vt_1,\frac{\vt_1}{2}}(\bom\times [0,T])$ for all $T>0$ and some $\vt_1\in (0,1)$,
  so that since (\ref{12.5}) implies that $\Theta\in C^{\vt_2,\frac{\vt_2}{2}}(\bom\times [0,T])$ for all $T>0$ and some
  $\vt_2\in (0,1)$, drawing on the inclusion $\gamma\in C^1([0,\infty))$ we obtain $\vt_3\in (0,1)$ such that
  $h\in C^{\vt_3,\frac{\vt_3}{2}}(\bom\times [0,T])$ for all $T>0$.
  We may hence apply standard parabolic Schauder theory (\cite{LSU}) to (\ref{12.12}) to see that for each $t_0>0$ and any $T>t_0$
  there exists $\vt_4=\vt_4(t_0,T)\in (0,1)$ such that $\Theta\in C^{2+\vt_4,1+\frac{\vt_4}{2}}(\bom\times [t_0,T])$.
  In particular, this entails that (\ref{12.12}) is actually satisfied in the classical sense, so that $(u,\Theta)$ indeed forms
  a strong solution of (\ref{0}) according to Definition \ref{dw}. \abs
  The estimate in (\ref{12.3}) now results from (\ref{10.3}) in a straightforward manner:
  Writing $\rho(t):=e^{\frac{\ks}{2} t}$ for $t\ge 0$, using (\ref{10.3}) we see that for $\zeps(x,t):=\rho(t) \weps(x,t)$,
  $(x,t)\in\bom\times [0,\infty)$, $\eps\in (0,\eps_0)$, we have
  \bas
	\io \zepsx^2(\cdot,t) \le k_6
	\qquad \mbox{for all $t>0$ and } \eps\in (0,\eps_0),
  \eas
  which due to (\ref{mass}) and a Poincar\'e inequality ensures the boundedness of the family $(\zeps)_{\eps\in (0,1)}$
  in $L^\infty((0,\infty);W^{1,2}(\Om))$.
  We can therefore extract a subsequence $(\eps_{j_i})_{i\in\N}$ of $(\eps_j)_{j\in\N}$ such that with some
  $z\in L^\infty((0,\infty);W^{1,2}(\Om))$ fulfilling
  \be{12.14}
	\io z_x^2(\cdot,t) \le k_6
	\qquad \mbox{for a.e.~} t>0
  \ee
  we have $\zeps \wsto z$ in $L^\infty((0,\infty);W^{1,2}(\Om))$ as $\eps=\eps_{j_i} \searrow 0$.
  Since, on the other hand, (\ref{12.6}) ensures that $\zeps\to \rho w=\rho u_{tt}$ in $C^0_{loc}(\bom\times [0,\infty))$
  as $\eps=\eps_j\searrow 0$, we can identify this limit according to $z(\cdot,t)=\rho(t)u_{tt}(\cdot,t)$ in $\Om$
  for a.e.~$t>0$, so that (\ref{12.14}) implies that
  \bas
	\io u_{xtt}^2 \le k_6 e^{-\ks t}
	\qquad \mbox{for a.e.~} t>0.
  \eas
  Along with two analogous arguments addressing $u_{xx}$ and $u_{xxt}$, this confirms (\ref{12.3}) whenever $C\ge 3k_6$.\abs
  In order to finally verify (\ref{12.4}) with some $\Theta_\infty>0$, we take $c_4>0$ such that
  \be{12.15}
	\bigg\| \psi - \frac{1}{|\Om|} \io \psi\bigg\|_{L^\infty(\Om)}
	\le c_4 \|\psi_x\|_{L^\infty(\Om)}
	\qquad \mbox{for all } \psi\in W^{1,\infty}(\Om),
  \ee
  and note that 
  \bas
	\io \Teps(\cdot,t) - \io \Teps(\cdot,t_0)
	= \int_{t_0}^t \frac{d}{ds} \bigg\{ \io \Teps(\cdot,s) \bigg\} ds
	\qquad \mbox{for all $t_0\ge 0$, $t>t_0$ and } \eps\in (0,1),
  \eas
  and that thus, thanks to (\ref{10.55}),
  \bas
	0 \le \io \Teps(\cdot,t) - \io \Teps(\cdot,t_0)
	\le k_6 \int_{t_0}^t e^{-\ks s} ds
	= \frac{k_6}{\ks} \cdot (e^{-\ks t_0} - e^{-\ks t})
  \eas
  for all $t_0\ge 0$, $t>t_0$ and $\eps\in (0,\eps_0)$. In line with (\ref{12.10}), this entails that
  \bas
	0 \le \io \Theta(\cdot,t) - \io \Theta(\cdot,t_0)
	\le \frac{k_6}{\ks} \cdot (e^{-\ks t_0} - e^{-\ks t})
	\qquad \mbox{for all $t_0\ge 0$ and } t>t_0,
  \eas
  hence implying that with some $\Theta_\infty\in (0,\infty)$ we have
  $\frac{1}{|\Om|} \io \Theta(\cdot,t) \nearrow \Theta_\infty$ as $t\to\infty$ and, in fact,
  \be{12.16}
	\Theta_\infty - \frac{1}{|\Om|} \io \Theta(\cdot,t_0)
	\le \frac{k_6}{\ks |\Om|} e^{-\ks t_0}
	\qquad \mbox{for all } t_0>0.
  \ee
  Since, again by (\ref{12.10}), we moreover know from (\ref{10.5}) that
  \bas
	\|\Theta_x(\cdot,t)\|_{L^\infty(\Om)}
	\le k_6 e^{-\ks t}
	\qquad \mbox{for all } t>0,
  \eas
  by combining (\ref{12.15}) with (\ref{12.16}) we infer that
  \bas
	\|\Theta(\cdot,t)-\Theta_\infty\|_{L^\infty(\Om)}
	&\le& \bigg\| \Theta(\cdot,t)-\frac{1}{|\Om|} \io \Theta(\cdot,t)\bigg\|_{L^\infty(\Om)}
	+ \bigg\| \frac{1}{|\Om|} \io \Theta(\cdot,t) - \Theta_\infty \bigg\|_{L^\infty(\Om)} \\
	&\le& c_4 \|\Theta_x\|_{L^\infty(\Om)}
	+ \bigg\| \frac{1}{|\Om|} \io \Theta(\cdot,t) - \Theta_\infty \bigg\|_{L^\infty(\Om)} \\
	&\le& c_4 k_6 e^{-\ks t}
	+ \frac{k_6}{\ks |\Om|} e^{-\ks t}
	\qquad \mbox{for all } t>0,
  \eas
  so that if we require $C$ to satisfy $C\ge c_4 k_6 + \frac{k_6}{\ks |\Om|}$, then indeed (\ref{12.4}) follows as well.
\qed

\bigskip

{\bf Acknowledgment.} \quad
The authors acknowledge support of the Deutsche Forschungsgemeinschaft (Project No. 444955436).
They moreover declare that they have no conflict of interest.\abs
{\bf Data availability statement.} \quad
Data sharing is not applicable to this article as no datasets were
generated or analyzed during the current study.

\small{

}

\end{document}